\documentclass[11 pt]{amsart}
\usepackage[applemac]{inputenc}

\usepackage[left=30mm,right=30mm,top=30mm,bottom=30mm]{geometry}
\usepackage[dvipsnames,svgnames,x11names,hyperref]{xcolor}
\usepackage[hyphens]{url}
\usepackage[colorlinks=true,linkcolor=Maroon,citecolor=blue,urlcolor=blue,hypertexnames=false,linktocpage]{hyperref}
\usepackage{bookmark}
\usepackage{amsmath,thmtools,mathtools}

\mathtoolsset{showonlyrefs=true}
\usepackage{blindtext}
\newcounter{mparcnt}

\usepackage{fancyhdr}
\usepackage{esint}
\usepackage{enumerate}

\usepackage{pictexwd,dcpic}
\usepackage{graphicx}
\usepackage{caption}
\usepackage[dvipsnames]{xcolor}
\swapnumbers
\declaretheorem[name=Theorem,numberwithin=section]{thm}
\declaretheorem[name=Remark,numberwithin=section,style=remark,sibling=thm]{rem}
\declaretheorem[name=Lemma,sibling=thm]{lemma}
\declaretheorem[name=Proposition,numberwithin=section,sibling=thm]{prop}
\declaretheorem[name=Definition,style=definition,numberwithin=section,sibling=thm]{defn}
\declaretheorem[name=Corollary,numberwithin=section,sibling=thm]{cor}

\declaretheorem[style=remark,name=Remark,numbered=no]{remark}

\numberwithin{equation}{section}

\newcommand{\ti}{\tilde}

\newcommand{\wh}{\widehat}

\newcommand{\cn}{\colon}
\newcommand{\sub}{\subset}
\newcommand{\ov}{\overline}
\newcommand{\mr}{\mathring}

\newcommand{\8}{\infty}

\newcommand{\be}{\beta}
\newcommand{\ga}{\gamma}
\newcommand{\de}{\delta}

\newcommand{\ka}{\kappa}
\newcommand{\la}{\lambda}
\newcommand{\om}{\omega}
\newcommand{\si}{\sigma}
\newcommand{\Si}{\Sigma}

\newcommand{\ze}{\zeta}

\newcommand{\Om}{\Omega}
\newcommand{\De}{\Delta}
\newcommand{\Ga}{\Gamma}

\newcommand{\La}{\Lambda}

\newcommand{\cD}{\mathcal{D}}

\newcommand{\cL}{\mathcal{L}}

\newcommand{\cN}{\mathcal{N}}
\newcommand{\cO}{\mathcal{O}}

\newcommand{\cS}{\mathcal{S}}

\newcommand{\cU}{\mathcal{U}}

\newcommand{\bD}{\bar{D}}

\newcommand{\bL}{\bar{L}}
\newcommand{\bM}{\bar{M}}
\newcommand{\bN}{\bar{N}}

\newcommand{\bX}{\bar{X}}
\newcommand{\bY}{\bar{Y}}

\newcommand{\del}{\partial}
\newcommand{\n}{\nabla}
\newcommand{\II}[2]{\mrm{II}\br{#1,#2}}
\newcommand{\fa}{\forall}

\newcommand{\ip}[2]{\left\langle #1,#2 \right\rangle}
\newcommand{\fr}[2]{\frac{#1}{#2}}
\newcommand{\tfr}[2]{\tfrac{#1}{#2}}
\newcommand{\x}{\times}

\DeclareMathOperator{\graph}{graph}

\DeclareMathOperator{\const}{const}

\DeclareMathOperator{\tr}{tr}

\DeclareMathOperator{\Rc}{Rc}

\newcommand{\pf}[1]{\begin{proof}#1 \end{proof}}
\newcommand{\eq}[1]{\begin{equation}\begin{alignedat}{2} #1 \end{alignedat}\end{equation}}

\newcommand{\br}[1]{\left(#1\right)}
\newcommand{\abs}[1]{\lvert #1\rvert}
\newcommand{\enum}[1]{\begin{enumerate}[(i)] #1 \end{enumerate}}

\newcommand{\ra}{\rightarrow}

\newcommand{\mt}{\mapsto}

\newcommand{\mrm}{\mathrm}

\newcommand{\hp}{\hphantom}
\newcommand{\q}{\quad}

\newcommand{\pard}[2]{\frac{\partial #1}{\partial #2}}

\title[Foliation of null cones by STCMC surfaces]{Foliation of null cones by surfaces of constant spacetime mean curvature near MOTS}

\author{Ben Lambert}
\address{School of Mathematics, The University of Leeds, Leeds, LS2 9JT, UK}
\email{\href{mailto:b.s.lambert@leeds.ac.uk}{b.s.lambert@leeds.ac.uk}}
\author{Julian Scheuer}
\address{Goethe-Universit\"at Frankfurt am Main, Institut f\"ur Mathematik, Robert-Mayer-Str.~10, 60325 Frankfurt, Germany} 
\email{\href{mailto:scheuer@math.uni-frankfurt.de}{scheuer@math.uni-frankfurt.de}}
\date{\today}
\keywords{Null geometry; Spacetime mean curvature; Null mean curvature flow; Foliation; Prescribed curvature.}

\begin{document}

\begin{abstract}
Marginally Outer Trapped Surfaces (MOTS) in spacetimes are well-known to indicate the existence of black holes. Using flow techniques, we prove that a neighbourhood of a stable MOTS in a null cone may be foliated by hypersurfaces of constant spacetime mean curvature. We also provide methods to construct prescribed spacetime mean curvature surfaces within null cones.
\end{abstract}

\maketitle

\section{Introduction}

In a spacetime $(\bar M^{n+2},\bar g)$, marginally outer trapped surfaces (MOTS) are defined by the property, that one of the null expansions is constantly zero, indicating that possibly after time reversion the light rays emanating from this surface are not visible from the outside. Under natural assumptions on the spacetime, the famous Hawking Penrose singularity theorems state that the existence of a MOTS yields the existence of a black hole and hence it is of interest to detect whether a spacetime admits them. Tod \cite{Tod:/1991} suggested using the mean curvature flow to find MOTS by flowing hypersurfaces in a time-symmetric spacelike $(n+1)$-slice, as the time-symmetry reduces the problem to find minimal surfaces within this slice. In the non time-symmetric case, Tod suggested the {\it null mean curvature flow} within the spacelike time-slice and this strategy was implemented by Bourni--Moore \cite{BourniMoore:/2019}, who defined a weak null mean curvature flow via a level-set approach and found weak versions, so-called {\it generalised MOTS}. All of the discussed approaches employ curvature flows within a spacelike slice of the spacetime, where for the null mean curvature flow version the flow speed is induced by the codimension 2 geometry of the flowing surface, namely from the null expansions coming from $\bar L$ and $L$, which form a normalised null pair of the normal bundle with $\bar g(L,\bar L)=1$. To make this precise, the Gauss equation of a spacelike $n$-surface $\Si\sub \bar M^{n+2}$ is given by
\eq{\label{Gaussian formula}\bar D_{X}Y = D_{X}Y - \chi(X,Y)\bar L - \bar\chi(X,Y)L.} 
Then the null mean curvature flow of Bourni--Moore is given by
\eq{\del_{t}x = -(\tr \chi) \nu,}
we $\nu$ is a normal of $\Si$ within the spacelike $(n+1)$-slice. 

The first smooth mean curvature flow to locate MOTS was invented by Roesch and second author \cite{RoeschScheuer:02/2022}, who employed a flow within the null hypersurface generated by a spacelike $2$-surface. To highlight the crucial difference to the flow by Bourni--Moore, the evolution equation is
\eq{\label{NMCF}\del_{t}x = -(\tr\chi) \bar L,}
i.e. the flow moves within a null hypersurface. This approach seems more natural as the flow speed is aligned with the flow direction  and indeed, under fairly mild assumptions on the spacetime, it is proved in \cite{RoeschScheuer:02/2022}, that this flow is able to detect MOTS smoothly.
Subsequently, Wolff has found interesting new properties of the flow \eqref{NMCF}, for example that in the standard null cone of Minkowski space, the induced metrics of the flowing surfaces move by Yamabe flow \cite{Wolff:04/2023}. 

The fact that \eqref{NMCF} defines a flow in null hypersurfaces which is able to detect MOTS in a spacetime, leads to the following natural question, which will be addressed in this paper:
\begin{center}Under which conditions can a neighbourhood of a MOTS in a null hypersurface be foliated by hypersurfaces of constant {\it spacetime mean curvature} (STCMC)? 
\end{center}
Here a $\lambda$-STCMC hypersurface satisfies

\eq{\abs{\vec H}^{2} = 2\theta\bar\theta = \la}
for some constant $\la$, and where 
\eq{\theta = \tr \chi,\q \bar\theta = \tr \bar\chi.}
Foliations by hypersurfaces of constant mean curvature in the ends of Riemannian manifolds have been extensively studied, and are fundamental to the famous definition of centre of mass in General Relativity given by Huisken--Yau \cite{HuiskenYau:/1996}, see for example \cite{BrendleEichmair:09/2014,EichmairKoerber:11/2024,Huang:12/2010,Ma:10/2011} and references therein (we do not attempt to include a complete bibliography here). In a semi-Riemannian context, under suitable hypotheses, global foliations by CMC hypersurfaces were shown to exist by Gerhardt \cite{Gerhardt:/2000,Gerhardt:09/2000, Gerhardt:/2006d, Gerhardt:12/1983}. The search for foliations of initial data sets by surfaces of constant spacetime mean curvature has received some attention in the recent years, as such foliations can also conveniently be used to define centres of mass for isolated systems. This was first observed by Cederbaum--Sakovich \cite{CederbaumSakovich:/2021} who proved the existence and uniqueness of such foliations for ends of asymptotically flat initial data sets under some structural assumptions. Kr\"oncke--Wolff \cite{KronckeWolff:12/2024} extended this to find $\lambda$-STCMC foliations on the ends of asymptotically Schwarzschildean light cones. Inspired by this use of STCMC surfaces, Huisken--Wolff \cite{HuiskenWolff:07/2025} defined an inverse spacetime mean curvature flow and constructed weak solutions.

Roughly stated, in this paper we answer the question above by proving the existence of a foliation by STCMC hypersurfaces of a null cone near a MOTS, provided the MOTS is stable in the sense of a suitable stability/Jacobi operator, see Definition \ref{def:stableetc} for details. In the following we state our main results, but for better readability occasionally refer to later sections for some precise definitions.

\subsection*{Main results}\addcontentsline{toc}{section}{Main results}

We state our main theorem.
 
\begin{thm}\label{thm:FoliationExistence}
Let $n\geq 2$ and $\bar N$ be a null cone on a stable MOTS $\Si_{0}\sub \bar N$ in a spacetime $(\bar M^{n+2},\bar g)$.
Define
\begin{flalign*}
\sigma &= \sup\{\kappa\geq 0\cn \text{There exists a strictly increasing smooth foliation of a future region of $\Si_{0}$}\\ 
&\qquad\qquad\qquad\text{by stable $\lambda$-STCMC hypersurfaces}~\text{with}~\la\in [0,\kappa]\}.
\end{flalign*}
Then the following statements hold.
\enum{
\item $\sigma>0$.
\item Either the foliation leaves every compact subset of $\bar N$, or $\limsup_{\lambda \ra \sigma}|A^{\Sigma_\lambda}|\ra \infty$,  or there is a smooth limit leaf $\Sigma_\sigma$ which is not stable. 
\item The foliation is unique in the sense that for any $\lambda$-STCMC surface $\widetilde{\Sigma}_\lambda$ with $0\leq \lambda<\sigma$ and $\widetilde{\Sigma}_\lambda\subset \cup_{0\leq \kappa < \sigma} \Sigma_\kappa$, there holds $\widetilde{\Sigma}_\lambda=\Sigma_\lambda$.
}
\end{thm}
\begin{remark}
Under reasonable conditions, we are able to show that curvature blow up as in part (ii) doesn't occur. See \autoref{rmk:foliationsize} for more details.
\end{remark}

We briefly explain some terminology used above. We call $\bar N$ a null cone on a MOTS $\Si_{0}$, if $\Si_{0}$ is a spacelike codimension 2 surface of $\bar M^{n+2}$ and a null basis $\{L,\bar L\}$ of the normal bundle can be chosen, such that $\bar L$ is future-directed, $L$ is past directed, \eqref{Gaussian formula} holds and
\eq{\theta:=\tr_{\Si_{0}}\chi = 0,\q \bar\theta := \tr_{\Si_{0}}\bar\chi >0.}%
  We give a detailed account in \autoref{sec:null geometry}.
  The MOTS $\Si_{0}$ is called {\it stable}, if there exists a positive function $f$ on $\Si_{0}$, such that 
  \eq{\cL_{\Si_{0}}f:=-\De f - 2\tau(\n f) +fB > 0,}
  where $\tau$ and $B$ are geometric quantities combined from extrinsic and intrinsic geometry of $\Si_{0}$ and $\bar N$, for details see \eqref{eq:stabilityquantity}. This notion of stability was inspired by the stability of a CMC hypersurface of Euclidean space, which in this setting would be an equivalent notion. Further similar notions of stability of MOTS were discussed in \cite{AnderssonMarsSimon:/2008, KronckeWolff:12/2024}.
  
  To understand the statement about the foliation and the definition of $\si$, we note that our null cones without loss of generality are of the form
  \eq{\bar N = [0,\La)\x \cS_{0},}
  where for the $s$-coordinate $s\in [0,\La)$, $\del_{s}$ is a null vector and $\cS_{0}$ is a compact base manifold, which in case of the above theorem may as well coincide with $\Si_{0}$. By increasing foliation we then mean, that all leaves $\Si_{\la}$ of the foliation are given by spacelike graphs 
  \eq{\Si_{\la} = \{(\om(z,\la),z)\cn z\in \Si_{0}\}}
   over $\Si_{0}$ and $\del_{\la}\om >0 $.

\begin{remark}
It is possible that the foliation may be extended to mean curvatures beyond $\sigma$. For example the (rotationally symmetric) STCMC slices of the standard null cone in Schwarzschild space, when written as a graph $s=\omega(z)$,  have spacetime mean curvature $|\vec{H}|^2(s)=s^{-2}(1-\frac{2M}{s})$. Here the MOTS is at $s=2M$ and this remains stable until $s=3M$.
\end{remark}

The strategy of the proof is to construct the leaves of the foliation by running curvature flows $(M_{t})_{t>0}$ in the null cone defined by
\eq{\del_{t}x = \br{\fr{\la}{2\bar\theta}-H}\bar L,}
where the {\it mean curvature of $M_{t}$}, $H = \theta_{|M_{t}}$ is as in \eqref{Gaussian formula} and the flow is started from a hypersurface $M_{0}$ in the future of $\Si_{0}$ which has strictly larger mean  curvature. This property is crucial to ensure the existence of barriers and the existence of $M_{0}$ is guaranteed by the stability of $\Si_{0}$. This $\la$-family of flows will then satisfy smooth estimates, which are uniform in time and $\la$. A few further arguments yield the desired foliation.

As a side product of our techniques we solve another geometric problem in null hypersurfaces. Crucially, our techniques described above do not seriously depend on the structure of the forcing term $\be=\fr{\la}{2\bar\theta}$. Instead we can allow very general $\be$  and this gives us the opportunity to solve the prescribed spacetime mean curvature problem.

We prove the following:

\begin{thm}\label{thm:PrescribedMeanCurvature} 
Let $n\geq 2$ and $\bar N$ be a null cone on $\cS_{0}$. Throughout $\bar N$, let $0\leq c_{\mr{\bar{\chi}}}, c_R, C_R$ satisfy
\begin{equation}|\mr{\bar{\chi}}|\leq c_{\mr{\bar{\chi}}}\bar{\theta}, \qquad 0\leq c_R \bar{\theta}^2\leq \ov{\Rc}(\bar{L},\bar{L})\leq C_R \bar{\theta}^2\label{eq:Nestimatedbytheta}.
\end{equation} 

Then there exist constants $C_0(n),C_1(n)\geq 0$, such that the inequality
 \eq{c_{\mr{\bar{\chi}}}+C_R<C_0+C_1c_R}
implies the existence of an explicit constant, $\mathcal{D}(n,c_R,c_{\mr{\bar{\chi}}}, C_R-c_R)$ which is smooth in its last three entries, where 
\eq{\mathcal{D}(n,c_R,0,0)= \tfrac{n^2}{(1+nc_R)^2}\left(\tfrac{(n-2)(n-10)}{4n^2} + \tfrac{n+6}{n}c_R+c_R^2\right),} 
such that the following holds:
Let $\rho\cn \bar{N}\ra \mathbb{R}$ be a smooth function, such that $\rho\geq\abs{\vec H}^{2}$ on $\cS_{0}$ and suppose there is a hypersurface $\Sigma^+$ to the future of $\cS_{0}$ with the property $\rho\leq|\vec{H}|^{2}$ on $\Sigma^+$. If $\mathcal{D}<0$, then we additionally suppose 
\eq{\sup_{\Sigma^+}\omega < \frac{n}{1+nc_R}(e^{\frac{\pi}{\sqrt{-\mathcal{D}}}}-1).}
Then, the flow
\[\dot{x}=\br{\fr{\rho}{2\bar\theta}-H}\bar{L} \]
starting from $\Sigma^+$ exists for all times and converges smoothly to a smooth prescribed mean curvature surface $\Sigma$ with
\eq{|\vec{H}|^2=\rho.}
\end{thm}

\begin{remark}
The case $c_{\mr{\bar{\chi}}}=0$ includes all rotationally symmetric null cones in rotationally symmetric spacetimes. Hence the geometric conditions above may be interpreted as the assumption that the ambient space is $C^2$ close to being rotationally symmetric. Furthermore, in the physically relevant cases where $n=2$ or $n=10$, the condition on $\sup_{\Sigma^+}\omega$ becomes vacuous in the rotationally symmetric situation and allows for large $\sup_{\Sigma^+}\omega$ if the ambient space is sufficiently close to rotationally symmetric. Finally we also note that for null hypersurfaces sufficiently close to rotationally symmetric hypersurfaces, increasing $c_{R}$ always weakens the hypotheses.
\end{remark}

\subsection*{Acknowledgments}
The authors would like to thank Wilhelm Klingenberg and Durham University for hosting various research visits where this work was initiated and completed. The second author would like to thank the first author and Leeds University for hosting a research visit, where parts of this work were written. The first author would like to thank the second author and Goethe--Universit\"at Frankfurt, where the proof of the first theorem was completed.

\section{Basic notions of null geometry}\label{sec:null geometry}

\subsection*{Null hypersurfaces}

We discuss special hypersurfaces of a Lorentzian manifold $(\bM^{n+2},\bar g)$, namely those which exhibit an everywhere degenerate induced metric. These are mostly referred to as {\it null hypersurfaces} in the literature. They carry some interesting and counterintuitive properties.

We denote by $\bD$ the Levi-Civita connection of $\bar g$ and $\bar\n$ denotes the gradient operator. For brevity we will mostly write $\ip{\cdot}{\cdot} = \bar g$.
Let us first recall, that every smooth hypersurface $\bN$ of a smooth manifold $\bM$ can locally be realised as the level set of a smooth function. Precisely, let $p\in \bN$, then there exists a neighbourhood $\cU\sub \bM$ of $p$ and a function $f\in C^{\8}(\cU)$, such that
\eq{\bN\cap \cU = \{f=0\}.}
This implies $T_{x}\bN = \ker df(x)$ and hence
\eq{(\bar\n f(x))^{\perp} = T_{x}\bN\q\fa x\in \bN\cap\cU. }
If $\bar{N}$ is a null hypersurface, the degeneracy of the induced metric $\bar g$ on $\bN$ implies that at every $x\in \bN$ there exists a nonzero null vector $\bar L\in T_{x}\bN$ with the property
\eq{\bar L^{\perp} = T_{x}\bar N.}
Hence we find
\eq{\bar\n f(x) \in T_{x}\bN,}
for otherwise the relation $\ip{\bar\n f(x)}{\bar L} = 0$ would imply $\bar L\in \ker \bar g$, which is impossible due to the non-degeneracy of $\bar g$. In addition, $\bar\n f$ and $\bar L$ are linearly dependent, which follows from the following lemma.

\begin{lemma}
On a vector space of dimension $n+1$ carrying a non-degenerate bilinear form $\bar g$ of signature $1$, every subspace $W$, such that $\bar g_{|W}=0$, is at most one-dimensional.
\end{lemma}

\pf{
By Sylvester's law of inertia, there exists an $n$-dimensional subspace $Z$ such that $\bar g_{|Z}$ is positive definite. Hence there holds $W\cap Z = \{0\}.$
However we also have
\eq{n+1\geq \dim(W+Z) = \dim W + n.}

}

 This behavior can be summarised by saying that the normal to a null hypersurface is also tangent. Under suitable assumptions as discussed later, this equips us with a well-defined, smooth global nonzero tangent field $\bar L$, which annihilates the tangent space. For $\bar X\in \Si^{\8}(\bN;T\bN)$, the latter being the space of smooth sections of the tangent bundle, we have
 \eq{0=\bar X \ip{\bar\n f}{\bar\n f}= {2}\ip{\bD_{\bX}\bar\n f}{\bar\n f}={2}\bar{D}^2 f(\bar{X},\bar{\n}f)&= {2}\bar{D}^2 f(\bar{\n}f,\bar{X})=2\ip{\bar X}{\bD_{\bar\n f}\bar\n f}  }
and it follows that $\bar D_{\bar\n f}\bar\n f$ is also a multiple of $\bar L$, which finally translates to
 \eq{\bar D_{\bar L}\bL = \ka\bL}
 for some function $\ka$. Thus, the flow generated by $\bL$ on $\bN$ is a flow of (pre-)geodesics and thus, in fact, the whole null hypersurface is ruled (i.e. foliated) by geodesics. In contrast to the non-degenerate case, the normal vector $\bL$ has no length and thus we cannot normalise it by division. Instead we use a different normalisation, namely $\bar L$ can be scaled to make $\ka = 0$. This follows from simple reparametrisation of the pregeodesics.
From now on we thus assume that $\bL$ has the property
\eq{\bD_{\bL}\bL = 0.}

The above flow construction yields the local structure of $\bN$.

\begin{prop}
Suppose $\bM$ is time-orientable, $\bar N\sub \bar M$ a null hypersurface, and let $\cS_{0}\sub\bN$ be a compact spacelike hypersurface. Then $\bN$  {\it splits around $\cS_{0}$}, i.e. there exists a diffeomorphism onto an open subset,
\eq{(\La_{-},\La_{+})\x\cS_{0}\ra \cN \sub\bN,}
where we use $s\in (\La_{-},\La_{+})$ as the canonical coordinate on this interval. Denoting by $\bar \ga$ the pullback metric on $\cN$ and by $\bar D^{\ast}$ the pullback connection, then $\del_{s}$ is a null vector and 
\eq{\bar D^{*}_{\del_{s}}\del_{s} = \bar D_{\bL}\bL = 0.}
Furthermore, $\bar L$ can be arranged future-directed. 
\end{prop}

\pf{
The splitting follows since on a compact $\cS_{0}$ we can choose a uniform time interval for the geodesic flow, and we simply discard the other part of $\bN$.
}

From now on, if $\bN$ splits around some $\cS_{0}$, we already simply assume that $\bN$ is a product as above. However, the property $\ka=0$ still does not determine $\bL$ uniquely, as we have one more degree of freedom in choosing the initial velocity of the geodesic flow. This can be accomplished by requiring that along $\cS_{0}$ we have
\eq{\bar\theta:=-\sum_{i=1}^{n} \ip{\bD_{e_{i}}e_{i}}{\bar L}\equiv 1, }
provided this quantity is nowhere zero for one (and hence any) choice of $\bL$, and where $(e_{i})_{1\leq i \leq n}$ is an orthonormal frame of $\cS_{0}$. In this case we say that $\bN$ is {\it locally a null cone around $\cS_{0}$}, meaning that $\bar\theta(s,\cdot)>0$ for all $s$ sufficiently small, where $\bar\theta(s,\cdot)$ is defined analogously via an orthonormal frame of the $s$-slice $\{s\}\x \cS_{0}$. Hence in the following, we simply assume this property to hold for $\bar N$ and discard the past of $\cS_{0}$, i.e. without loss of generality be the null cone $\bar N$ identified with
\eq{\bar N \cong [0,\La)\x\cS_{0}}
and say that $\bar N$ is a {\it{null cone on $\cS_{0}$}.}

We then often refer to this splitting as the {\it canonical background foliation} of the null cone.
This fixing of $\bL$ can be used to unambiguously define the second fundamental form of $\bN$.

\begin{defn}
Suppose $\bN$ is a null cone on $\cS_{0}$. Then we define the {\it second fundamental form} of $\bar N$ by
\eq{\bar\chi(\bX,\bY) = \ip{\bar D_{\bX}\bL}{\bY}\q\fa \bX,\bY\in \Si^{\8}(\bN;T\bN).}
\end{defn}

\begin{rem}
Even if $\bN$ is not a null cone, we can define a second fundamental form. However, in this case the object is only defined up to multiplication by smooth functions on $\cS_{0}$.
\end{rem}

\subsection*{Spacelike graphs}
The splitting structure is well suited to describe spacelike graphs. Suppose that $M\sub \bN$ is a spacelike graph given by
\eq{M = \{x(z)=(\om(z),z)\cn z\in \cS_{0}\}.}
We recall the Gauss formula for vector fields $X,Y$ on $M$ and decompose the normal part $\II{X}{Y}$ conveniently: Let $D$ be the Levi-Civita connection of the metric $g = x^{\ast}\bar g$. Then
\eq{\label{eq:null Gauss}\bD_{Y}X = D_{Y}X + \II{X}{Y}=D_{Y}X - h(X,Y)\bL -  \bar\chi(X,Y) \nu,}
where the past-directed vector $\nu$ complements $\bar L$ in such a way that
\eq{\ip{\nu}{\nu} = 0,\q \ip{\nu}{\bar L} = 1,}
and where we, as commonly done, identify $X\in \Si^{\8}(M;TM)$ with its pushforward $x_{\ast}X$. 
We say that $\{\nu,\bar L\}$ is a {\it null pair} for $M$. 
Note that this terminology even makes sense if $M$ does not factor through a null cone a priori.

\begin{defn}
For a graph $M$ as given above, we define the {\it second fundamental form of $M$ in $\bN$} to be the bilinear form $h$. We also define 
\eq{H:=\tr_{g}h}
as the {\it mean curvature} of $M$ in $\bN$.
\end{defn}

The trace of $\bar\chi$ with respect to $g$ is more subtle.

\begin{lemma}
Let $M$ and $\ti M$ be two spacelike graphs, which intersect at a point $x_{0}\in \bN$. Then there holds
\eq{\tr_{g}\bar\chi|_{x_0} = \tr_{\ti g}\bar\chi|_{x_0}.} 
\end{lemma}

\pf{
Here and in the following, for the differential of $\om$ we adopt the notation
\eq{\om_{i}:=\del_{i}\om.}
Fix a local coordinate frame $(\del_{i})$ on $\cS_{0}$ and write $x\cn \cS_{0}\ra M$ and $\ti x\cn \cS_{0}\ra \ti M$ for the embeddings of the graphs. Then $x_{i} := x_{\ast}\del_{i}$ and likewise for $\ti x$ give rise to coordinate representations $g_{ij}$ and $\ti g_{ij}$.
There holds
\eq{g_{ij} = \ip{x_{i}}{x_{j}} = \ip{\om_{i}\bar L + \del_{i}}{\om_{j}\bar L + \del_{j}}=\ip{\ti\om_{i}\bar L + \del_{i}}{\ti\om_{j}\bar L + \del_{j}} = \ti g_{ij},}
where we used that $\bar L$ annihilates everything in $T\bN$.
In addition there holds
\eq{\bar\chi_{ij}:=\bar\chi(x_{i},x_{j}) = \bar\chi(\om_{i}\bar L + \del_{i},\om_{j}\bar L+ \del_{j}) = \bar \chi(\ti x_{i},\ti x_{j}),}
since
\eq{\label{eq:barchibarL}\bar\chi(\bar L,\bar X) = \ip{\bD_{\bL}\bL}{\bX} = 0 \q\fa \bX\in \Si^{\8}(\bN;T\bN).}  
}

The following definition is hence sensible:

\begin{defn}
\enum{
\item
Suppose $\bar N$ is a null cone $\cS_{0}$. Then we define the mean curvature of $\bN$ in $\bM$ to be the function $\bar\theta\in C^{\8}(\bN)$ defined by 
\eq{\bar\theta(s,z):=\tr_{g_{s}}\bar\chi,}
where $g_{s}$ is the induced metric of the spacelike set $\{(s,z)\cn z\in \cS_{0}\}$.
\item For the constant graphs $\{s=\const\}$ we also reserve the notation $L$ for $\nu$, $\chi$ for $h$ and $\theta$ for $H$. 
}
\end{defn}

\begin{rem}
\enum{
\item
We also use the notation $L$, $\chi$ and $\theta$ in case that a spacelike submanifold is given in absence of a factorising null cone. The notation $\nu$, $h$ and $H$ is reserved to distinguish general graphs within a null cone from the respective quantities of the coordinate slices of the background foliation.
\item Since the quantity $\chi$ is defined on every leaf of the background foliation, it can be viewed as an element of $\Si^{\8}(\bar N;T^{0,2}\bar N)$, via the formula
\eq{\chi(\bar X,\bar Y) = \ip{\bar X}{\bar D_{\bar Y}L}.}
Where for vector fields on $\cS_{0}$ this tensor is the second fundamental of the $s$-slice $\{\om\equiv s\}$, it is left to identify its action on the $\bar L$ directions,
\eq{\label{eq:extendedchi}\chi(\bar L,\bar L) &=\ip{\bar L}{\bar D_{\bar L}L}= -\ip{\bar D_{\bar L}\bar L}{L} = 0,\\
 \chi(\bar L,\partial_i) &= -\ip{\bar D_{\partial_i}\bar L}{L} = -\zeta(\partial_i),\\
  \chi(\partial_i,\bar L) &= \ip{\partial_i}{\bar D_{\bar L}L}= -\ip{\bar D_{\bar L}\partial_i}{L} = -\ip{\bar D_{\partial_i}\bar L}{L} = -\zeta(\partial_i). }
  }
\end{rem}

Here we have introduced a new linear form. As it is common for higher codimensional submanifolds, we have to take {\it torsion} into account. Contrary to the hypersurface case, where the derivative of the normal is always tangent, in higher codimension there might be components in other normal directions. 
Hence we define two more quantities, the {\it torsion of the null cone} and the {\it torsion of a spacelike graph $M$},
\eq{\label{eq:defszetatau} \zeta(\bar X) &= \ip{\bD_{\bX}\bL}{L}\q\fa \bX\in \Si^{\8}(\bN;T\bN),\\
\tau(\bar X) &= \ip{\bar D_{\bar X}\bar L}{\nu}\q \fa \bar X\in \Si^{\8}(\cS_0;x^{*}T\bar N).}

\begin{rem}\label{rem:H2}
From \eqref{eq:null Gauss} we immediately obtain that the mean curvature vector of a spacelike graph $M\sub \bM$, which factors through a null cone $M\ra \bN\sub\bM$, is given by
\eq{\vec H = - H\bar L -  \bar\theta \nu }
and hence
\eq{\abs{\vec H}^{2} = 2H\bar\theta.}
\end{rem}

\begin{defn}
Within a null cone on $\cS_{0}$ we call a spacelike hypersurface a MOTS if $H=0$.
\end{defn}

\subsection*{A connection on the local null cone}
We have already introduced the tensor $\bar\chi \in \Si^{\8}(\bar N,T^{0,2}\bar N)$ and we will introduce several other such tensors later. As we have to differentiate them and as there is no canonical Levi-Civita connection on $\bar N$ induced from $\bar M$ due to the degeneracy of the induced metric, we introduce ad hoc a connection which makes use of the canonical background foliation of a null cone.

\begin{defn}
Suppose $\bar N$ is a null cone on $\cS_{0}$. Then we define
\eq{\hat D\cn \Si^{\8}(\bar N;T\bar N)\x \Si^{\8}(\bar N;T\bar N)&\ra \Si^{\8}(\bar N;T\bar N)\\
				(\bar X,\bar Y)&\mt \hat D_{\bar X}\bar Y = \bar D_{\bar X}\bar Y - \ip{\bar D_{\bar X}\bar Y}{\bar L}L. }
This connection is readily extended to arbitrary tensors by the standard Leibniz rule.
\end{defn}

\subsection*{Comparison formulae for spacelike graphs}
Using the parametrisation $x(z)=(\omega(z),z)$ and a local coordinate frame $\{\del_{i}\}$ to $\cS_{0}$, we have
\eq{x_i = \del_i+ \om_{i}\bar{L},}
which immediately implies
\eq{\ip{L\circ x}{x_i}&=\om_i  \label{eq:projL}.}
We let $\abs{\cdot}$ denote the norm induced on the graphs.
There holds
\begin{equation}
    \nu = L + \tfr 12|\n \om |^2 \bar{L} - \n \om = L-\tfr 12 \abs{\n \om}^{2}\bar L - \om^{i}\del_{i}, \label{eq:nuidentities}
\end{equation}
since we immediately observe $\ip{\bar{L}}{\nu}=1$, $\ip{\nu}{x_i}=\om_{i}-\om_{i}=0$ and $\ip{\nu}{\nu} = 0$ as required.
Putting these facts together we obtain
\eq{\label{eq:omegaij}D_{x_{j}x_{i}}\om = \del_{j}\om_{i} - \Ga^{k}_{ij}\om_{k} &= \ip{\bar D_{x_{j}}L}{x_{i}} + \ip{\nu - \tfr 12\abs{\n\om}^{2}\bar L}{\bar D_{x_{j}}x_{i}} = \chi_{ij} - h_{ij}+u\bar\chi_{ij},}
where we define
\eq{u:=- \ip{\nu}{L}=\tfrac 1 2|\n\om|^2\geq 0, \label{eq:defu}}
which is a geometric function capturing the gradient of our graphs.

 We note that
\begin{equation}\label{eq:tauasbackground}\tau(\bar X)=\ip{\bar{D}_{\bar X} \bar L}{L+\tfr 12|\n\om|^2\bar{L} -\n \om}=\zeta(\bar X)-\bar{\chi}(\bar X, \n \om)
\end{equation}
and also record the following identities for later use, where we use \eqref{eq:extendedchi} and \eqref{eq:nuidentities},
\eq{ \ip{\nu}{D_{\bar{L}}L} & = \zeta(\n \om)\label{eq:nuDbarLL},}
\eq{
\ip{\nu}{\bar{D}_{\bar X}L}&=\ip{u \bar{L}-\n \om}{\bar{D}_{\bar X}L}=-u\zeta(\bar X)-\chi(\bar X, \n \om)\q\fa \bar X\in \Si^{\8}(\bar N;T\bar N).\label{eq:nuDiL}
}

Finally, we need a lemma that relates the graph function with the connection on $\bar N$.

\begin{lemma}\label{lemma:connection}
For $T\in \Si^{\8}(\bar N;T^{0,1}\bar N)$ there holds
\eq{D_{x_{i}}T(x_{j}) = \hat D_{x_{i}}T(x_{j}) + T(\bar\chi_{ij}\n\om - (u\bar\chi_{ij} + h_{ij})\bar L).}
\end{lemma}

\pf{We also denote by $T$ restriction of $T$ to the spacelike graph $M = \graph\om$.
\eq{D_{x_i}T(x_j)&=x_{i}( T(x_j))-T(D_{x_i}x_j)\\
&=\hat{D}_{x_i} T(x_j)+T(\bar{D}_{x_i}x_j - \ip{\bar{D}_{x_i}x_j}{\bar{L}}L-D_{x_i}x_j)\\
&=\hat{D}_{x_i} T(x_j)+T(-\bar{\chi}_{ij}\nu-h_{ij}\bar{L} + \bar{\chi}_{ij}L)\\
&=\hat{D}_{x_i} T(x_j)+T(\bar{\chi}_{ij}(L-\nu)-h_{ij}\bar{L})\\
&=\hat{D}_{x_i} T(x_j)+T(\bar{\chi}_{ij}(\n \om-\tfr 12|\n \om |^2 \bar{L} )-h_{ij}\bar{L})\\
&=\hat{D}_{x_i} T(x_j)+T(\bar{\chi}_{ij}\n \om-(u\bar{\chi}_{ij}+h_{ij})\bar{L}).
}
}

We have obvious similar identities for higher order tensors.

\section{Evolution equations}

\subsection{Evolution equations for general speeds}
We study general evolutions of the form
\begin{equation}
\dot{x}(t,z)= f(t,z) \bar L(x(t,z)), \label{eq:genspeed}
\end{equation}
where
\eq{x\cn [0,T^{*})\x \cS_{0}\ra \bar N}
is a flow of graphs and a dot indicates the time derivative.

\begin{lemma}\label{lem:genevolomega}
Suppose that $x$ satisfies \eqref{eq:genspeed} in a null cone on $\cS_{0}$. Then the evolution of $\om$ is
\eq{\dot{\om} - \De\om = f - \theta + H - u\bar\theta.}
Let $\widetilde{\omega}$ be any smooth function on $\bar N$ with $\bar{L}(\widetilde{\omega})=0$. If $x$ satisfies \eqref{eq:genspeed} then
\eq{\dot{\tilde{\omega}}-\Delta\tilde{\omega} = -g^{ij}\hat{D}_{\del_i \del_j}\tilde{\omega}-\omega^i(\bar{\theta}\delta_i^k-2\bar{\chi}_{i}^{k})\tilde{\omega}_{k}.}
\end{lemma}
\begin{proof}
The evolution of $\om$ follows immediately from \eqref{eq:omegaij}. 

We have $\dot{\tilde{\omega}}=0$. On the other hand, using \autoref{lemma:connection}, we have 
\eq{D_{x_i x_j}\tilde{\omega} = \hat{D}_{x_i x_j}\tilde{\omega}+\bar{\chi}_{ij}\omega^k\tilde{\omega}_{k}.}
Note that $\hat{D}_{\bar{L}}\bar{L}=0$ and 
\eq{\hat{D}_{\del_i}\bar{L}=\bar{\chi}^{k}_{i}\del_k+\zeta_{i}\bar{L},} so $\hat{D}_{\bar{L}\bar{L}}\tilde{\omega}=0$ and 
\eq{\hat{D}_{\del_i \bar L}\ti\om = \hat{D}_{\bar{L} \del_i }\tilde{\omega}=-\bar{\chi}^{k}_{i}\tilde{\omega}_{k}.} Therefore,
\eq{\De \ti\om = g^{ij} D_{ij}\tilde{\omega} = g^{ij}(\hat{D}_{\del_i \del_j}\tilde{\omega}-\bar{\chi}^{k}_{i}\tilde{\omega}_{k}\omega_j-\bar{\chi}^{k}_{j}\tilde{\omega}_{k}\omega_i+\bar{\chi}_{ij}\omega^k \tilde{\omega}_{k})=g^{ij}\hat{D}_{\del_i \del_j}\tilde{\omega}+\omega^i(\bar{\theta}\delta_i^k-2\bar{\chi}_{i}^{k})\tilde{\omega}_{k}.}
\end{proof}

\begin{lemma}\label{lem:ddtnu}
Suppose that $x$ satisfies \eqref{eq:genspeed} in a null cone on $\cS_{0}$. Then
\eq{\label{eq:ddtnu}\bar D_{\dot x} \nu = - \n f -f \tau^{k}x_k}
and
\eq{\label{eq:ddtugeneral}\dot{u} = \ip{\n f}{L} - f\bar\chi(\n\om,\n\om).}
\end{lemma}
\begin{proof}
    We have 
    \eq{\ip{\bar D_{\dot x} \nu}{x_i}=-\ip{\nu}{\bar D_{\dot x} x_i}=-\ip{\nu}{\bar D_{x_i}\dot x}=-\ip{\nu}{\bar D_{x_i}( f \bar{L})}=-df(x_i)-f\tau_i,}
    \eq{\ip{\bar D_{\dot x} \nu}{\bar{L}}=-\ip{\nu}{\bar D_{\dot x} \bar{L}}=0}
    and
    \eq{ 0=\frac{d}{dt}\ip{\nu}{\nu}=\ip{\bar D_{\dot x} \nu}{\nu}.}
    The first statement now follows. For $\dot u$ we compute
    \eq{\dot u = \ip{\n f + f\tau^{k}x_{k}}{L} - \ip{\nu}{\bar D_{\dot{x}}L} &= \ip{\n f}{L} + f\tau(\n\om) - f\ze(\n\om)=\ip{\n f}{L} - f\bar\chi(\n\om,\n\om),}
where we used \eqref{eq:extendedchi}, \eqref{eq:tauasbackground} and \eqref{eq:nuDbarLL}.
\end{proof}

For the following evolution equations we need to clarify our convention on the curvature tensor. We use the one from \cite{ONeill:/1983}, 
\eq{R(X,Y)Z = D_Y (D_X Z)-D_X(D_Y Z)+D_{[X,Y]}Z,}
which in coordinates reads
\eq{R^{m}_{lkj}x_{m} = R(x_{k},x_{j})(x_{l})  .}
We obtain the Gauss equation, \cite[p.~100, Thm.~5]{ONeill:/1983}
\eq{R^{m}_{lkj} &= \bar R^{m}_{lkj}+ \bar g(\mrm{II}_{kl},\mrm{II}^{m}_{j}) - \bar g(\mrm{II}_{jl},\mrm{II}^{m}_{k})=\bar R^{m}_{lkj} + \bar\chi_{kl}h^{m}_{j} + h_{kl}\bar\chi^{m}_{j} - \bar\chi_{jl}h^{m}_{k} - h_{jl}\bar\chi^{m}_{k},}
where $\mrm{II}$ is the full second fundamental form.

\begin{lemma}\label{lem:ddthijgij}
Suppose that $x$ satisfies \eqref{eq:genspeed} in a null cone on $\cS_{0}$. Then
\eq{\label{eq:genfddtgij}\dot{g}_{ij} = 2f\bar{\chi}_{ij},}
\eq{\label{eq:genfddthij}\dot{h}_{ij} 
=- D_{ij}f-df(x_j)\tau_i-df(x_i)\tau_j+f(\bar{\chi}_j^kh_{ki}-D_i\tau(x_j)-\tau_i\tau_j-\ip{\bar{R}(x_i,\bar L)x_j}{\nu})
} 
and
\eq{\label{eq:genfddtH}\dot{H} 
&=- \Delta f-2\tau(\n f)-f(\bar{\chi}^{ij}h_{ij}+D^i\tau_i+|\tau|^2+\ov{\Rc}(\bar{L}, \nu)+\ip{\bar{R}(\nu,\bar{L})\nu}{\bar{L}}).}
\end{lemma}
\begin{proof}
We compute
\eq{\dot{g}_{ij} = \ip{\bar D_{\dot{x}}x_i}{x_j}+\ip{x_i}{\bar D_{\dot{x}} x_j} = \ip{\bar D_{x_i}(f\bar{L})}{x_j}+\ip{x_i}{\bar D_{x_j}(f\bar{L})}=2f\bar{\chi}_{ij}.}
Using \autoref{lem:ddtnu}, we compute
\begin{flalign*}
\dot{h}_{ij} &= -\del_{t}\ip{\bar{D}_{x_i}x_j}{\nu}\\
&= -\ip{\bar{D}_{\dot{x}}(\bar{D}_{x_i}x_j)}{\nu} -\ip{\bar{D}_{x_i}x_j}{\bar{D}_{\dot{x}}\nu}\\
&= -\ip{\bar{D}_{x_i}(\bar{D}_{x_j}\dot{x})+\bar{R}(x_i,\dot{x})x_j}{\nu} -\ip{\bar{D}_{x_i}x_j}{\bar{D}_{\dot{x}}\nu}\\
&= -\ip{\bar{D}_{x_i}(df(x_{j})\bar{L} +f(\bar{\chi}_j^kx_k+\tau_{j}\bar{L}))}{\nu}-\ip{\bar{R}(x_i,\dot{x})x_j}{\nu} -\ip{\bar{D}_{x_i}x_j}{\bar{D}_{\dot{x}}\nu}\\
&= -\del_i(df(x_{j}))-df(x_{j})\tau_i-df(x_{i})\tau_j-f\ip{\bar{D}_{x_i}(\bar{\chi}_j^kx_k+\tau_{j}\bar{L})}{\nu}\\
&\qquad-\ip{\bar{R}(x_i,\dot{x})x_j}{\nu} -\ip{\bar{D}_{x_i}x_j}{\bar{D}_{\dot{x}}\nu}\\
&= -\del_i(df(x_{j}))-df(x_{j})\tau_i-df(x_{i})\tau_j+f\bar{\chi}_j^kh_{ki}-f\del_i(\tau_{j})-f\tau_i\tau_j\\
&\qquad-\ip{\bar{R}(x_i,\dot{x})x_j}{\nu} -\ip{\bar{D}_{x_i}x_j}{\bar{D}_{\dot{x}}\nu}\\
&=-D_{ij} f-df(x_{j})\tau_i-df(x_{i})\tau_j+f\bar{\chi}_j^kh_{ki}-fD_i\tau_{j}-f\tau_i\tau_j-f\ip{\bar{R}(x_i,\bar L)x_j}{\nu},
\end{flalign*}
which gives the second equation. To take the trace, we note that
\begin{flalign*}
\ip{\bar{R}(x_i,\bar{L})x^i}{\nu}&=\ov{\Rc}(\bar{L}, \nu) - \tfrac{1}{2}\ip{\bar{R}(\bar{L}+\nu,\bar{L})(\bar{L}+\nu)}{\nu} + \tfrac{1}{2}\ip{\bar{R}(\bar{L}-\nu,\bar{L})(\bar{L}-\nu)}{\nu}\\
&=\ov{\Rc}(\bar{L}, \nu) - \tfrac{1}{2}\ip{\bar{R}(\nu,\bar{L})(\bar{L}+\nu)}{\nu} - \tfrac{1}{2}\ip{\bar{R}(\nu,\bar{L})(\bar{L}-\nu)}{\nu}\\
&=\ov{\Rc}(\bar{L}, \nu) - \ip{\bar{R}(\nu,\bar{L})\bar{L}}{\nu}.
\end{flalign*}
Using this identity, along with the evolution of $g_{ij}$ and $h_{ij}$, we see 
\eq{
\dot{H} 
&=- \Delta f-2\tau(Df)+f(-\bar{\chi}^{ij}h_{ij}-D^i\tau_i-|\tau|^2-\ov{\Rc}(\bar{L}, \nu)-\ip{\bar{R}(\nu,\bar{L})\nu}{\bar{L}}).
}\end{proof}

\subsection{Evolution equations for PMCF}

We define the Prescribed Mean Curvature Flow (PMCF) to be 
\eq{\label{eq:PMCF}\dot{x}=(\beta-H)\bar{L}}
where $\beta\cn\bar N\rightarrow \mathbb{R}$ is a smooth function.
We start with a corollary of \autoref{lem:genevolomega}.

\begin{lemma}\label{lem:evolomegaPMCF}
Under \eqref{eq:PMCF}, the evolution of $\omega$ is
\eq{\dot{\om} - \De\om = \beta - \theta - u\bar\theta.}
For $\widetilde{\omega}$ any smooth function on $\bar{N}$ with $\bar{L}(\widetilde{\omega})=0$, the following evolution equation holds:
\eq{\dot{\tilde{\omega}}-\Delta\tilde{\omega} = -g^{ij}\hat{D}_{\del_i \del_j}\tilde{\omega}-\omega^i(\bar{\theta}\delta_i^k-2\bar{\chi}^{k}_{i})\tilde{\omega}_{k}.}
\end{lemma}

Our proof is based on $C^{1}$-estimates. Hence we need evolution equations up to that order.

\begin{lemma}
Under \eqref{eq:PMCF}, the evolution of $u$ is
\eq{
\dot{u}-\Delta u	&=\be_{i}\om^{i}-\be\bar\chi_{ij}\om^{i}\om^{j} -u\bar\theta_{k}\om^{k}+2\zeta_{j}u^{j}-g^{ij}(\hat{D}_{\n\om}\chi)_{ij} -4u\bar{\chi}_{kj}\zeta^j\om^{k}-2\chi_{kj}\zeta^j\om^{k}\\
	&\hp{=} -\bar\theta u \bar\chi_{ij}\om^{i}\om^{j}-\bar\theta\chi_{ij}\om^{i}\om^{j}+2u\bar\chi_{lj}\bar\chi^{l}_{k}\om^{j}\om^{k} - 2\bar\chi_{ij}\om^{i}u^{j} - \abs{D^{2}\om}^{2} - \bar R^{i}_{lki}\om^{l}\om^{k}.
}

\end{lemma}
\begin{proof}
For this proof, if we furnish the functions $H$, $\be$, $u$ or $\om$ by indices, we mean covariant differentiation with respect to $D$, e.g.
\eq{\om_{ij}dx^{i}dx^{j} = D_{x^{j}}(\om_{i}dx^{i})dx^{j}.}
All coordinates are taken with respect to $(x_{i})_{1\leq i\leq n}$ and lifting of indices happens with respect to $g$.
    Using equation \eqref{eq:ddtugeneral}, we obtain
    \eq{
    \dot{u}&=\ip{\n (\beta-H)}{L} - (\beta-H) \bar\chi(\n\om,\n\om)
    =-H_{i}\om^{i} +\be_{i}\om^{i}-(\beta-H)\bar\chi(\n\om,\n\om).
    }
We recall 
$2u =  \om^k\om_k$
and compute with the help of \eqref{eq:omegaij},
\eq{u_{i} = \om_{ki} \om^{k} =\om^k(u\bar{\chi}_{ik}+\chi_{ik}-h_{ik}).}
We obtain
\eq{
u_{ij} &= g_{mi} {\om^{m}}_{kj} \om^{k} + \om_{ki}  \om^{k}_{j}\\
	&=g_{mi}{\om^{m}}_{jk}\om^{k} + g_{mi}R^{m}_{l kj}\om ^{l}\om^{k} + \om_{ki}\om^{k}_{j} \\ 
	&=\om_{ijk}\om^{k} + (\bar\chi_{lk}h_{ij} + h_{lk}\bar\chi_{ij} - \bar\chi_{jl}h_{ik} - h_{jl}\bar\chi_{ik})\om^{l}\om^{k} + g_{mi}\bar R^{m}_{lkj}\om^{l}\om^{k} + \om_{ki}\om^{k}_{j}.
}
Tracing with respect to $g$ gives
\eq{
\Delta u&= (\De\om)_{k}\om^{k} + H\bar\chi_{ij}\om^{i}\om^{j} +\om^k h_{k}^l(\bar{\theta}\om_l-2\bar{\chi}_{lj}\om^j) + \abs{D^{2}\om}^{2} + \bar R^{i}_{lki}\om^{l}\om^{k}\\
&= (\De\om)_{k}\om^{k} + H\bar\chi_{ij}\om^{i}\om^{j} +(u\om^k\bar{\chi}_k^l+\om^k\chi_k^l-u^{l})(\bar{\theta}\om_l-2\bar{\chi}_{lj}\om^j) + \abs{D^{2}\om}^{2} + \bar R^{i}_{lki}\om^{l}\om^{k}\\
&= (\De\om)_{k}\om^{k} + H\bar\chi_{ij}\om^{i}\om^{j} +\bar\theta u \bar\chi_{ij}\om^{i}\om^{j}+\bar\theta\chi_{ij}\om^{i}\om^{j}-\bar\theta u^{i}\om_{i}\\
	&\hp{=}-2\bar\chi_{lj}\om^{j}(u\om^{k}\bar\chi^{l}_{k} + \om^{k}\chi^{l}_{k}) + 2\bar\chi_{ij}\om^{i}u^{j} + \abs{D^{2}\om}^{2} + \bar R^{i}_{lki}\om^{l}\om^{k}.
}
We further expand the first term with the help of \autoref{lemma:connection}, \eqref{eq:extendedchi} and \eqref{eq:omegaij}:
\eq{
(\De\om)_{k}\om^{k}&= (u\bar{\theta}+ g^{ij}\chi_{ij} - H)_{k}\om^{k}\\
	&=\bar\theta u_{k}\om^{k} + u\bar\theta_{k}\om^{k} + g^{ij}(\hat D_{\n\om}\chi)_{ij}+2\chi(\bar\chi_{k}^{j}\n\om-(u\bar\chi_{k}^{j}+h_{k}^{j})\bar L,x_{j})\om^{k} - H_{k}\om^{k}\\
&=\bar\theta u_{k}\om^{k} + u\bar\theta_{k}\om^{k}+g^{ij}(\hat{D}_{\n\om}\chi)_{ij} + 2 \om^k\om^l\chi^j_l\bar{\chi}_{jk}+2\om^k(u\bar{\chi}_{kj}+h_{kj})\zeta^j-H_{k}\om^{k}\\
&=\bar\theta u_{k}\om^{k} + u\bar\theta_{k}\om^{k}+g^{ij}(\hat{D}_{\n\om}\chi)_{ij} + 2 \om^k\om^l\chi^j_l\bar{\chi}_{jk}+2(2u\om^k\bar{\chi}_{kj}+\om^k\chi_{kj}-u_j)\zeta^j\\
	&\hp{=}-H_{k}\om^{k}\\
&=\bar\theta u_{k}\om^{k} + u\bar\theta_{k}\om^{k}-2\zeta_{j}u^{j}+g^{ij}(\hat{D}_{\n\om}\chi)_{ij} + 2 \chi^j_l\bar{\chi}_{jk}\om^{k}\om^{l}+4u\bar{\chi}_{kj}\zeta^j\om^{k}+2\chi_{kj}\zeta^j\om^{k}\\
	&\hp{=}-H_{k}\om^{k}\\
}
and combining these equalities we get
\eq{
\De u&= u\bar\theta_{k}\om^{k}-2\zeta_{j}u^{j}+g^{ij}(\hat{D}_{\n\om}\chi)_{ij} +4u\bar{\chi}_{kj}\zeta^j\om^{k}+2\chi_{kj}\zeta^j\om^{k}-H_{k}\om^{k} + H\bar\chi_{ij}\om^{i}\om^{j}\\
	&\hp{=} +\bar\theta u \bar\chi_{ij}\om^{i}\om^{j}+\bar\theta\chi_{ij}\om^{i}\om^{j}-2u\bar\chi_{lj}\bar\chi^{l}_{k}\om^{j}\om^{k} + 2\bar\chi_{ij}\om^{i}u^{j} + \abs{D^{2}\om}^{2} + \bar R^{i}_{lki}\om^{l}\om^{k}\ .
}
The claim follows from combining with \autoref{lem:ddtnu}
and cancellation of terms involving $H$.
\end{proof}

We also recall Raychaudhuri's optical equation, which we need in the sequel:

\begin{prop}\label{prop:optical}
On $\bar{N}$, 
\eq{
\bar L\bar\theta =-\abs{\mr{\bar\chi}}^{2} - \tfr{1}{n}\bar\theta^{2} - \ov{\Rc}(\bar L,\bar L).
}
\end{prop}

\pf{
\eq{
\bar L\bar\theta &= \bar L (g^{ij}\ip{\bar D_{\del_{i}}\bar L}{\del_{j}})\\
			&= -g^{ik} (\ip{\bar D_{\bar L}x_{k}}{x_{l}} + \ip{x_{k}}{\bar D_{\bar L}x_{l}})g^{lj}\ip{\bar D_{\del_{i}}\bar L}{\del_{j}}) + g^{ij}\ip{\bar D_{\bar L}\bar D_{\del_{i}}\bar L}{\del_{j}} + g^{ij}\ip{\bar D_{\del_{i}}\bar L}{\bar D_{\del_{j}}\bar L}\\
			&=-\abs{\bar\chi}^{2} + g^{ij}\ip{\bar R(\del_{i},\bar L)\bar L}{\del_{j}}\\
			&=-\abs{\bar\chi}^{2} - \ov{\Rc}(\bar L,\bar L) - \tfr 12 \ip{\bar R(\bar L + L,\bar L)\bar L}{\bar L + L} + \tfr 12 \ip{\bar R(\bar L - L,\bar L)\bar L}{\bar L - L} \\
			&=-\abs{\mr{\bar\chi}}^{2} - \tfr{1}{n}\bar\theta^{2} - \ov{\Rc}(\bar L,\bar L).
}
}

\begin{cor}\label{cor:ev u}
The evolution of $u$ satisfies the estimate
\eq{
\dot{u} - \De u&= 2\zeta_{i}u^{i} - 2\bar\chi_{ij}\om^{i}u^{j} -\abs{D^{2}\om}^{2}+ \cO(u^{\fr 32}) +\tfr{4}{n^{2}}\bar\theta^{2}u^{2}-2u^{2}\ov{\Rc}(\bar L,\bar L)\\
	&\hp{=}+ 2u^{2}\abs{\mr{\bar\chi}}^{2}+(\tfr{4}{n}-1)\bar\theta u\mr{\bar\chi}_{kj}\om^{k}\om^{j}+2u\mr{\bar\chi}_{lj}\mr{\bar\chi}^{l}_{k}\om^{j}\om^{k},
}
where $\cO(r)$ is characterised by the estimate
\eq{|\mathcal{O}(r)|\leq C(1+|\beta|+|\bar{L}\beta|+|d\beta|)(1+|r|)}
where $C$ is a constant which is bounded while the flow remains in any compact set.
\end{cor}

\pf{
We group all terms in the evolution of $u = \tfr 12 \abs{\n \om}^{2}$ according to their order. Noting that $x_i=\del_i+\om_i \bar{L}$, there holds for every function $f$ on the graph $M$ 
\eq{\n f = g^{ij}f_{j}x_{i} = f^{i}\del_{i} + f^{i}\om_{i}\bar L,}
and in particular for the graph function $\om$ itself we get
\eq{\n \om = \om^{i}\del_{i} + 2u\bar L.}
Then there holds, using \eqref{eq:barchibarL}, \eqref{eq:extendedchi} and \eqref{eq:defszetatau},
\eq{
\be^{i}\om_{i}-2\chi_{kj}\zeta^{j}\om^{k}  &= \n\om(\be)-2\chi_{kj}\zeta^{j}\om^{k}   = \cO(u),\\
-\be\bar\chi_{ij}\om^{i}\om^{j} &= \cO(u).
}
Then we have
\eq{\label{pf:cor u 1}-u\bar\theta_{k}\om^{k} = - u\n\om(\bar\theta) = -2u^{2}\bar L(\bar\theta) + \cO(u^{\fr 32}).  
}

The next relevant terms are
\eq{-g^{ij}(\hat D_{\n\om}\chi)_{ij} - 4u\bar\chi_{kj}\zeta^{j}\om^{k}&=-g^{ij}(\om^{k}\hat D_{\del_{k}}\chi)(x_{i},x_{j}) - 2ug^{ij}(\hat D_{\bar L}\chi)(x_{i},x_{j})+\cO(u^{\fr 32})=\cO(u^{\fr 32}),\\
}
since $\hat D_{\bar L}\chi(\bar L,\bar L) = 0$, due to \eqref{eq:extendedchi}, and where we also used \eqref{eq:defszetatau}.
The next two second order terms are
\eq{-\bar\theta u\bar\chi_{ij}\om^{i}\om^{j} + 2u\bar\chi_{lj}\bar\chi^{l}_{k}\om^{j}\om^{k}&=-\bar\theta u(\mr{\bar\chi}_{ij} + \tfr{\bar\theta}{n}g_{ij})\om^{i}\om^{j} + 2u(\mr{\bar\chi}_{lj} + \tfr{\bar\theta}{n}g_{lj})(\mr{\bar\chi}^{l}_{k} + \tfr{\bar\theta}{n}\de^{l}_{k})\om^{j}\om^{k}\\
		&=(\tfr{4}{n^{2}}-\tfr 2n)\bar\theta^{2}u^{2}+(\tfr{4}{n}-1)\bar\theta u\mr{\bar\chi}_{kj}\om^{k}\om^{j}+2u\mr{\bar\chi}_{lj}\mr{\bar\chi}^{l}_{k}\om^{j}\om^{k}.
}
We continue, using \eqref{eq:extendedchi},
\eq{
    -\bar{\theta}\chi_{ij}\om^{i}\om^{j}&=-\bar{\theta}\om^i\om^j\chi(\del_i,\del_j)+4u\bar{\theta}\om^i\zeta_i = \cO(u^{\fr 32}).
}
For the final term involving the Riemann tensor, we compute by completing the basis of $T\bar{M}$ using $\frac{1}{\sqrt{2}}(\bar{L}\pm L)$,
\eq{
\ip{-\bar{R}(x_j, \n\om)x^j}{\n \om}&=-\ov{\Rc}(\n \om, \n \om)+\tfrac{1}{2}\ip{\bar{R}(\bar{L}+L, \n\om)(\bar{L}+L)}{\n \om}\\
	&\hp{=}-\tfrac{1}{2}\ip{\bar{R}(\bar{L}-L, \n\om)(\bar{L}-L)}{\n \om}\\
&=-\ov{\Rc}(\n \om, \n \om)+\ip{\bar{R}(L, \n\om)\bar{L}}{\n \om}+\ip{\bar{R}(\bar{L}, \n\om){L}}{\n \om}\\
&=-\ov{\Rc}(\n \om, \n \om)+2\ip{\bar{R}(L, \n\om)\bar{L}}{\n \om}\\
&=-4u^2\ov{\Rc}(\bar{L}, \bar{L})-2u\om^i\ov{\Rc}(\del_i, \bar{L})-\om^i\om^j\ov{\Rc}(\del_i, \del_j)\\
	&\hp{=}+2\om^i\om^j\ip{\bar{R}(L, \del_i)\bar{L}}{\del_j}+4\om^ju\ip{\bar{R}(L, \bar{L})\bar{L}}{\del_j}\\
	&=-4u^{2}\ov{\Rc}(\bar L,\bar L) + \cO(u^{\fr 32}).
}

Plugging everything together and also using \autoref{prop:optical} gives the result.
}

We will make use of test functions to obtain estimates for $u$. The following lemma reduces requirements on the test function to an ordinary differential inequality.

\begin{lemma}\label{lem:GradEstTestFunctionHelperLemma}
Suppose that $\phi = u \mu^{- 2}(\om)$ for some $\mu\cn\mathbb{R}\ra\mathbb{R}^+$. Then, at any positive maximum of $\phi$,
\eq{\dot{\phi} - \De\phi
	&\leq 4\mu \phi^2[ \Phi(x)\mu +\Psi\mu'+\mu'']+ C_\mu \cO(\phi^{\fr 32})}
where
\begin{align*}
\Phi&=\tfrac{1}{n^2}\bar\theta^2 - \tfrac 1 2\ov{\Rc}(\bar{L},\bar{L})+\tfrac 1 2|\mr{\bar\chi}|^2+\tfrac {4-n} {4n}\bar{\theta}\mr{\bar{\chi}}_{kj}\om^k\om^ju^{-1}+\tfrac 12\mr{\bar\chi}_{lj}\mr{\bar\chi}^l_{k}\om^j\om^ku^{-1},\\
\Psi&=\tfrac{n-4}{2n}\bar\theta- {\mr{\bar\chi}}_{ij}\om^{i}\om^{j}u^{-1}.
\end{align*}
and $C_\mu=C_\mu(|\mu|_{C^1},\inf \mu)$. 

For $\tilde{\omega}$ as in \autoref{lem:evolomegaPMCF}, define $\tilde{\phi} = u \mu^{- 2}(\om-\tilde{\om})$, then at any maximum of $\tilde{\phi}$,
\eq{\dot{\tilde\phi} - \De\tilde\phi
		&\leq 4\mu \tilde\phi^{2}[ \Phi(x)\mu+\Psi\mu' +\mu'']+\tilde{C}_\mu\widetilde{\cO}(\tilde{\phi}^{\frac 32})\\
	}
where $\tilde{C}_\mu=\tilde{C}_\mu(|\mu|_{C^2},\inf \mu)$ and $\widetilde{\cO}(r)$ which is characterised by the estimate
\eq{\abs{\widetilde{\cO}(r)}\leq \widetilde{C}(1+\abs{\beta} + \abs{\bar L\be} + \abs{d\be})(1 +|r|)}
where $\widetilde{C}$ is a constant which is a constant which is bounded while the flow remains in any compact set, but also may depend on $\hat{D}$-derivatives of $\widetilde{\omega}$ up to second order.
\end{lemma}

\pf{
The function
\eq{\phi = u \rho(\om)}
satisfies
\eq{\dot{\phi} - \De\phi&=\rho (\dot{u}-\De u) + u\rho'(\om)(\dot\om - \De\om) - u\rho''(\om)\abs{D\om}^{2}-2\ip{\n u}{\n \rho}\\
		&\leq \rho u[2\zeta_{i}(\log u)^{i} - 2\bar\chi_{ij}\om^{i}(\log u)^{j} -u^{-1}\abs{D^{2}\om}^{2}+ \cO(u^{\fr 12}) +4\Phi(x)u\\
	&\hp{=}\qquad + \rho^{-1}\rho'(\om)( \be - \theta   - u\bar\theta) - 2\rho^{-1}\rho''(\om)u-2\rho^{-1}\rho'\ip{\n \log u}{\n\omega}].\label{eq:Evolphifirst}}

At a maximum we have $D \log u = -\rho^{-1}\rho' D\om$. Additionally we have
\begin{equation}
|D u|^2=|\om^i\om_{i}^kx_k|^2\leq |D\om|^2|D^2\om|^2=2u |D^2\om|^2,\label{eq:PropertyFromMaximality}
\end{equation}
so at a maximum,
\eq{u^{-1}|D^2 \om |^2 \geq \tfrac 1 2 |D \log u|^2 = \tfrac 1 2\rho^{-2}( \rho')^2|D\omega|^2=\rho^{-2}( \rho')^2u.}
Hence, at a maximum,
\eq{\dot{\phi} - \De\phi
		&\leq \rho u[-2\rho^{-1}\rho'\zeta^i\omega_i+ 2\rho^{-1}\rho'\bar\chi_{ij}\om^{i}\om^{j} -\rho^{-2}(\rho')^2u+ \cO(u^{\fr 12}) +4\Phi(x)u\\
	&\hp{=}\qquad - \rho^{-1}\rho'(\om)\bar\theta u - 2\rho^{-1}\rho''(\om)u+2\rho^{-2}(\rho')^2\abs{D\omega}^2+\rho^{-1}\rho'(\beta-\theta)]\\
	&= \rho u^2[ 4\Phi(x)+\rho^{-1}\rho'\left(\tfrac{4-n}{n}\bar\theta+ 2{\mr{\bar\chi}}_{ij}\om^{i}\om^{j}u^{-1}\right)  - 2\rho^{-1}\rho''+3\rho^{-2}(\rho')^2]\\
	&\hp{=}\qquad+ \cO(u^{\fr 32})(\rho +|\rho'|).}

We now set $\mu=\frac{1}{\sqrt{\rho}}$ and note that
\eq{-2\mu' = \frac{\rho'}{\rho^{\frac 3 2}}=\mu\rho^{-1}\rho', \qquad 4\mu'' = -2 \frac{\rho''}{\rho^\frac 3 2}+3 \frac{(\rho')^2}{\rho^\frac 5 2} = \mu(3\rho^{-2}(\rho')^2-2\rho^{-1}\rho''),}
so
\eq{\dot{\phi} - \De\phi
	&\leq \rho\mu^{-1} u^2[ 4\Phi(x)\mu -2\left(\tfrac{4-n}{n}\bar\theta+ 2{\mr{\bar\chi}}_{ij}\om^{i}\om^{j}u^{-1}\right)\mu'+4\mu'']+ \cO(u^{\fr 32})(\mu^{-2}+|\mu'|\mu^{-3})\\
	&= 4\mu \phi^2[ \Phi(x)\mu +\Psi\mu'+\mu'']+ \cO(u^{\fr 32})(\mu^{-2}+|\mu'|\mu^{-3}).}
The claim now follows as $\cO(u^\frac 32)\leq (1+\mu^{3})\cO(\phi^\frac 32)$.

We now repeat the above computation but with $\tilde{\phi}=u\rho(\omega-\tilde\omega)$. \autoref{lem:evolomegaPMCF} implies
\eq{\dot{\tilde{\omega}}-\Delta \tilde{\omega} = \widetilde{\mathcal{O}}(u^\frac{1}{2}),}
and hence for $\Omega=\omega-\tilde{\omega}$ we obtain
\eq{\dot{\Omega}-\Delta\Omega = \dot{\omega}-\Delta \omega +\widetilde{\mathcal{O}}(u^\frac{1}{2})\ .}
At a maximum of $\tilde{\phi}$, $D \log u = -\rho^{-1}\rho' D\Om$, and so using the previous evolution and \autoref{cor:ev u},
\eq{\dot{\tilde\phi} - \De\tilde\phi&=\rho (\dot{u}-\De u) + u\rho'(\Om)(\dot\Om - \De\Om) - u\rho''(\Om)\abs{\n\Om}^{2}-2\ip{\n u}{\n \rho}\\
		&\leq \rho u[2\zeta_{i}(\log u)^{i} - 2\bar\chi_{ij}\om^{i}(\log u)^{j} -u^{-1}\abs{D^{2}\om}^{2}+ \cO(u^{\fr 12}) +4\Phi(x)u\\
	&\hp{=}\qquad + \rho^{-1}\rho'( \be - \theta   - u\bar\theta +\widetilde\cO(u^{\frac 12})) - \rho^{-1}\rho''\abs{\n\Om}^{2}-2\rho^{-1}\rho'\ip{\n \log u}{\n\Omega}]\\
	&= \rho u[-2\rho^{-1}{\rho}'\zeta_{i}\Om^{i} + 2\rho'\rho^{-1}\bar\chi_{ij}\om^{i}\Om^{j} -u^{-1}\abs{D^{2}\om}^{2} +4\Phi(x)u\\
	&\hp{=}\qquad + \rho^{-1}\rho'( \be - \theta   - u\bar\theta ) - \rho^{-1}\rho''\abs{\n\Om}^{2}+2\rho^{-2}(\rho')^{2}|\n \Om|^{2}+(1+\rho^{-1}|\rho'|)\widetilde\cO(u^{\frac 12})]\\
	&= \rho u[2\rho^{-1}\rho'\bar\chi_{ij}\om^{i}\om^{j} -u^{-1}\abs{D^{2}\om}^{2} +4\Phi(x)u -\rho^{-1}\rho'\bar\theta u\\
	&\hp{=}\qquad - \rho^{-1}\rho''\abs{\n\Om}^{2}+2\rho^{-2}(\rho')^{2}|\n \Om|^{2}+(1+\rho^{-1}|\rho'|)\widetilde\cO(u^{\frac 12})].\\
}
As
\eq{|\n\Omega|^2=|\n\omega|^2+2\omega^i\tilde{\omega}_{i}+\tilde{\omega}_{i}g^{ij}\tilde{\omega}_{j}=|\n\omega|^2+\widetilde{\mathcal{O}}(u^\frac 1 2),}
and
\eq{u^{-1}|D^2\omega|^2\geq \tfrac 1 2 \rho^{-2}(\rho')^2|\n\Omega|^2=\rho^{-2}(\rho')^2u+\rho^{-2}(\rho')^2\widetilde{\mathcal{O}}(u^\frac 1 2),}
we see that
\eq{\dot{\tilde\phi} - \De\tilde\phi&\leq	 \rho u[ 2\rho'\rho^{-1}\bar\chi_{ij}\om^{i}\om^{j} -\rho^{-2}(\rho')^{2}u +4\Phi(x)u-\rho^{-1}\rho'\bar\theta u\\
	&\hp{=}\qquad  - 2\rho^{-1}\rho''u+4\rho^{-2}(\rho')^{2}u+(1+\rho^{-1}|\rho''|+\rho^{-2}|\rho'|^{2}+\rho^{-1}|\rho'|)\widetilde\cO(u^{\frac 12})]\\
	&= \rho u^2\left[  4\Phi(x)+\rho'\rho^{-1}\left(\tfrac{4-n}{n}\bar\theta+2\mr{\bar\chi}_{ij}\om^{i}\om^{j}u^{-1}\right)  - 2\rho^{-1}\rho''+3\rho^{-2}(\rho')^{2}\right]\\
	&\hp{=}\qquad +(\rho+|\rho'|+|\rho''|+\rho^{-1}|\rho'|^{2})\widetilde\cO(u^{\frac 32}).
	}

Substituting $\mu=\tfrac{1}{\sqrt{\rho}}$, we have
\[\rho'=\frac{-2\mu'}{\mu^{3}}, \qquad \frac{(\rho')^{2}}{\rho}=\frac{4(\mu')^{2}}{\mu^{4}}, \qquad \rho''=\frac{-2\mu''}{\mu^{3}}+6\frac{(\mu')^{2}}{\mu}\]
and so
\eq{\dot{\tilde\phi} - \De\tilde\phi
		&\leq 4\mu \tilde\phi^{2}[ \Phi(x)\mu+\Psi\mu' +\mu'']\\
	&\hp{=}\qquad  +\left(\mu^{-2}+|\mu'|\mu^{-3}+|\mu'|^{2}\mu^{-4}+|\mu'|^{2}\mu^{{-1}}+|\mu''|\mu^{-3}\right)\widetilde{\cO}(u^{\frac 32}).
	}
	The claim now follows as previously.
}

Finally, our convergence results rely on monotonic movement of the flow, hence we require the following evolution which immediately follows from \eqref{eq:genfddtH}.
\begin{lemma}
Under \eqref{eq:PMCF}, the evolution of the speed $f=\beta-H$ is given by
\eq{\label{eq:evolf}\dot{f} &=\Delta f+2\tau(\n f)+f(\bar{L}(\beta)+\bar{\chi}^{ij}h_{ij}+D^i\tau_i+|\tau|^2+\Rc(\bar{L}, \nu)+\ip{\bar{R}(\nu,\bar{L})\nu}{\bar{L}}).
}
\end{lemma}

\section{Spacetime CMC foliations near a MOTS and proof of \autoref{thm:FoliationExistence}}\label{sec:foliation}

We now demonstrate that there exists a foliation under suitable assumptions. First we recall the definition of a {\it spacetime constant mean curvature surface} (STCMC), which has been studied recently in several papers, e.g. \cite{CederbaumSakovich:/2021,HuiskenWolff:07/2025}.

\begin{defn}\label{def:stableetc}
A hypersurface $\Si$ of $\bar{N}$ is a $\lambda$-STCMC surface if on $\Si$,
\eq{|\vec{H}|^2 = 2H \bar{\theta}=\lambda.} 
We say a smooth foliation made up of graphs 
\eq{\Si_{\xi}:=\{(\om(z,\xi),z)\cn z\in \Si\}} above $\Si$ is strictly increasing if $\omega_{\xi}>0$ everywhere.	

We say that a $\lambda$-STCMC surface is stable if there is a smooth function $f>0$ on $\Sigma$ with $\mathcal{L}_{\Sigma} f>0$ where
\eq{\mathcal{L}_\Sigma f=- \Delta f-2\tau(\n f)+fB}
and
\begin{equation}\label{eq:stabilityquantity}B:=-\mr{\bar{\chi}}^{ij}h_{ij}-\tfrac \lambda 2(\bar{\theta}^{-2}\abs{\mr{\bar\chi}}^{2} +\tfr{2}{n}+\bar{\theta}^{-2}\ov{\Rc}(\bar L,\bar L)) -D^i\tau_i-|\tau|^2-\ov\Rc(\bar{L}, \nu)-\bar{g}(\bar{R}(\nu,\bar{L})\nu,\bar{L})).
\end{equation}
\end{defn}

We observe the following lemma.
\begin{lemma}\label{lemma:increasing vec H}
A compact $\lambda$-STCMC hypersurface $\Sigma$ is stable if and only if there is locally a smooth strictly increasing foliation of hypersurfaces $\Sigma_\xi$ above $\Sigma$ (expressed as graphs $\omega(z,\xi)$) which have $\frac{\partial}{\partial \xi}\Big|_{\xi=0}|\vec{H}_\xi|^2>0$.
\end{lemma}

\begin{proof}
By equation \eqref{eq:genfddtH} and \autoref{prop:optical}, for a general variation $\dot{x}=f\bar{L}$ we have

\eq{
\frac{d}{dt}{|{\vec{H}}|^2}&=2fH\bar{L}(\bar\theta)+2\bar\theta\dot{H}= 2\bar{\theta}[- \Delta f-2\tau(\n f)+fB]
}
where
\eq{
B&=-\bar{\chi}^{ij}h_{ij}-H\bar{\theta}^{-1}\abs{\mr{\bar\chi}}^{2} - H\tfr{1}{n}\bar\theta -D^i\tau_i-|\tau|^2-\Rc(\bar{L}, \nu)-\bar{g}(\bar{R}(\nu,\bar{L})\nu,\bar{L})- H\bar{\theta}^{-1}\ov{\Rc}(\bar L,\bar L))\\
&=-\mr{\bar{\chi}}^{ij}h_{ij}-H(\bar{\theta}^{-1}\abs{\mr{\bar\chi}}^{2} +\tfr{2}{n}\bar\theta+\bar{\theta}^{-1}\ov{\Rc}(\bar L,\bar L)) -D^i\tau_i-|\tau|^2-\Rc(\bar{L}, \nu)-\bar{g}(\bar{R}(\nu,\bar{L})\nu,\bar{L})).
}

Given a strictly increasing foliation of hypersurfaces above $\Sigma$ then we may set $f=\omega_\xi|_{\xi=0}>0$. Then, using the above (swapping the parameter $t$ with $\xi$),
\[0<\frac{\partial}{\partial \xi}\Big|_{\xi=0}|\vec{H}_\xi|^2=2\bar{\theta}[- \Delta f-2\tau(\n f)+fB] \]
so $\Sigma$ is stable.

On the other hand, given a stable $\lambda$-STCMC hypersurface $\Sigma$ with graph function $\omega^\Sigma(z)$, then we may define a strictly increasing foliation by $\omega(z,\xi)=\omega^\Sigma(z)+\xi f(z)$. This has $\omega_{\xi}=f$ so by the above computation $\frac{\partial}{\partial \xi}\Big|_{\xi=0}|\vec{H}_\xi|^2>0$. 
\end{proof}
\begin{remark}$\ $
\enum{
\item Clearly, a sufficient condition for stability is the condition $B>0$. In this case we may simply take $f=1$. In practice, this is easier to verify.
\item In general to have a positive function satisfying the above, the maximum principle implies that it is necessary that $B>0$ somewhere (for the function to have a positive minimum), but in general $B>0$ will not be a necessary condition.
}
\end{remark}

Our method for proving \autoref{thm:FoliationExistence} is to flow to $\lambda$-STCMC hypersurfaces by the flow
\begin{equation}
\dot{x} = (\tfrac{\lambda}{2} \bar{\theta}^{-1}-H)\bar{L},
\label{eq:CMCF}
\end{equation}
which is a special case of \eqref{eq:PMCF} with $\beta=\tfrac{\lambda}{2}\bar{\theta}^{-1}$. We define the {\it prescription} of the flow to be the constant $\lambda$.

We begin with several simple consequences of the maximum principle. Suppose that $\Sigma_1$ and $\Sigma_2$ are defined by graph functions $\omega_1$ and $\omega_2$. We say $\Sigma_2$ is above $\Sigma_1$ if $\omega_2\geq\omega_1$. We say $\Sigma_2$ is strictly above $\Sigma_1$ if $\omega_2>\omega_1$.

\begin{lemma}\label{lem:ellipticSMP}
Suppose that $\Sigma_1$ and $\Sigma_2$ have mean curvature vectors $\vec H_1$ and $\vec H_2$ respectively, so that $\Sigma_2$ is above $\Sigma_1$ and, considered as functions in the graphical parametrisations, $|\vec H_2|^2\geq |\vec H_1|^2$. Then either $\Sigma_1=\Sigma_2$ or $\Sigma_1$ and $\Sigma_2$ are disjoint.
\end{lemma}

\begin{proof}
From \eqref{eq:omegaij}, the mean curvature is related to the graph function by 
\begin{flalign*}
|\vec{H}(z)|^2&=2\bar{\theta}(\omega(z),z)\theta(\omega(z),z)+|\n \om(z)|^2\bar\theta^2(\omega(z),z)-2\bar{\theta}(\omega(z),z)\Delta \omega(z)]\\
&=-b(\widetilde{D} \omega, \omega, z)-a^{ij}(\omega,z)\widetilde{D}^2_{ij}\omega \ ,
\end{flalign*}
where we write $\widetilde{D}$ for the Levi-Civita connection on $\mathcal{S}_0$, we note that $b$ is some smooth function and $a^{ij} = 2\bar{\theta}(\omega(z),z)g^{ij}(\omega,z)$ is positive definite.

We compute
\begin{flalign*}
0&\leq |\vec{H}_2|^2-|\vec{H}_1|^2\\
&=-[b(\tilde{D} \omega_2, \omega_2, z)-b(\tilde{D} \omega_1, \omega_1, z)+a^{ij}(\omega_2,z)\tilde{D}_{ij} \omega_2 -a^{ij}(\omega_1,z)\tilde{D}_{ij} \omega_1 ]\\
&=-\mathcal{E}(\omega_2-\omega_1)
\end{flalign*}
where $\mathcal{E}=\hat{a}^{ij}\tilde{D}_{ij}+\hat{b}^i\tilde{D}_i +\hat{c}$ and
\begin{align*}
\hat{a}^{ij} &= a^{ij}(\omega_2,z)\\
\hat{b}^{i} &= \int_0^1\pard{b}{p^i}(\tau \widetilde{D}\omega_2+(1-\tau)\widetilde{D}\omega_1, \tau \omega_2+(1-\tau)\omega_1,z)d\tau \\
\hat{c}&= \int_0^1\pard{b}{\omega}(\tau \widetilde{D}\omega_2+(1-\tau)\widetilde{D}\omega_1, \tau \omega_2+(1-\tau)\omega_1,z)d\tau \\
&\qquad +\widetilde{D}_{ij} \omega_1 \int_0^1\pard{a^{ij}}{\omega}( \tau \omega_2+(1-\tau)\omega_1,z)d\tau \ .
\end{align*}

Set $\Omega=\omega_2-\omega_1$. Hence, $\Omega\geq 0$ and for a uniformly elliptic operator, $\mathcal{E}\Omega\leq 0$. The Lemma now follows from the strong maximum principle.
\end{proof}

\begin{lemma}\label{lem:SMPonomega}
Suppose that for $i\in \{1,2\}$, $x_i$ satisfies \eqref{eq:CMCF} with prescription $\lambda_i$, for $\lambda_1\leq\lambda_2$. If the corresponding graph of $x_2$ is above the one of $x_1$ at the initial time, then this property holds at all later times. If $\lambda_1<\lambda_2$ then the graph of $x_2$ is strictly above the one of $x_1$ at all positive times.
\end{lemma}
\begin{proof}
As seen in \autoref{lem:evolomegaPMCF}, for a flow $x = \mrm{graph}~\om$ satisfying \eqref{eq:CMCF} with prescription $\la$, the graph function $\om$ satisfies
\eq{\dot\om = \fr{\la}{2\bar\theta} - \theta - u\bar\theta = F(\ti D^2 \om,\ti D\om,\om,\cdot),}
where $\ti D$ is the connection on $\cS_0$. Let $\om_i$, $i=1,2$, be the graph functions corresponding to $x_i$, then the function $\om_2-\om_1$
satisfies a linear equation with locally bounded coefficients, as can be seen from a computation similar to the one in \autoref{lem:ellipticSMP}. The first statement follows from the standard parabolic maximum principle, as for example in \cite[Lemma~2.3]{Lieberman:/1998}.

For the second statement, note that we already know $\om_2\geq \om_1$ everywhere. At a hypothetical point $(t,z)$, at which $\Om$ is zero, there holds
\eq{
0 &\geq \partial_t(\omega_2-\omega_1)(t,z)\\
&\geq\tfrac {\lambda_2}2 \bar\theta^{-1}(\omega_2,\cdot)-\tfrac {\lambda_1}2 \bar\theta^{-1}(\omega_1,\cdot)-\theta(\omega_2,\cdot)+\theta(\omega_1,\cdot)-\tfrac 12 |\n \om_2|^2\bar{\theta}(\om_2,\cdot)+\tfrac 12 |\n \om_1|^2\bar{\theta}(\om_1,\cdot)\\
 &= \tfrac12(\lambda_2-\lambda_1)\bar\theta^{-1}(\omega_2,\cdot)> 0,
}
which is a contradiction.
The claim now follows.
\end{proof}

We now provide a local uniqueness of $\lambda$-STCMCs which are sufficiently close to a stable $\kappa$-STCMC.

\begin{lemma}\label{lem:StabilitytoCMCunique}
Suppose that $\widehat{\Sigma}$ is a stable $\kappa$-STCMC with graph function $\widehat{\omega}$ and a smooth positive function $f$ be given for which $\mathcal{L}_{\widehat{\Sigma}} f>0$. Then there is a constant $\epsilon=\epsilon(\bar{N}, \widehat{\Sigma},f)>0$ such that for any $\lambda\geq \kappa$ and any $\lambda$-STCMC hypersurface $\Sigma$ with graph function $\omega$, which lies above $\widehat{\Sigma}$ and satisfies $|\widehat{\omega}-\omega|_{C^2}<\epsilon$, is unique.
\end{lemma}

\begin{proof}
For $f$ as in the statement, we consider the stability operator on the perturbed manifolds $\Sigma_{\widehat{\omega}+\varphi f}$ whose graphs are given by $\widehat{\omega}+\varphi f$ for some smooth  $\varphi:\widehat{\Sigma}\ra \mathbb{R}$. Then 
\eq{\mathcal{L}_{\Sigma_{\widehat{\omega}+\varphi f}}f = Q_f(\cdot,\varphi,D\varphi,D^2\varphi)}
for some smooth $Q_f$ where $Q_f(\cdot,0,0,0)>0$. By compactness and continuity there exists a $\delta>0$ such that for any ${\omega}$ with $|\widehat{\omega}-\omega|_{C^2}<\delta$ we have $\mathcal{L}_{\Sigma_\omega} f>0$ and so 
\eq{\frac{d}{d\xi}\Big|_{\xi=0}|\vec{H}_{\Sigma_{{\omega}+\xi f}}|^2>0.\label{pf:StabilityCMCunique-1}}

Suppose there are two solutions $\omega_1$ and $\omega_2$ above $\widehat{\Sigma}$ both of which have constant spacetime mean curvature $\lambda\geq\kappa$ and
\eq{|\widehat{\omega}-\omega_i|_{C^2}<\epsilon:=\tfrac{\delta}{16(1+|f|_{C^2})(1+|f^{-1}|_{C^2})}}
 for $i=1,2$. Suppose for a contradiction that \eq{0<M:=\max((\om_2-\om_1)f^{-1}).} 
Then, we know that $M<\tfrac{\delta}{8|f|_{C^2}}$ and we define the function $\omega(\cdot,\xi):= \omega_1+ \xi f $ for $\xi\in[0,M]$. We note that
\eq{|\omega(\cdot, \xi)-\widehat{\omega}|_{C^2}<|\omega_1-\widehat{\omega}|_{C^2}+\xi|f|_{C^2}\leq \tfrac{\delta}{16}+\tfrac{\delta}{8}\leq \tfrac \delta 4.}

In particular,  $\Sigma_{\omega(\cdot,M)}$ has spacetime mean curvature
\eq{|\vec{H}_{\omega(\cdot,\xi)}|^2=\lambda+ \int_0^M\frac{d}{d\xi}|\vec{H}_{\omega(\cdot,\xi)}|^2d\xi> \lambda,}
where we used \eqref{pf:StabilityCMCunique-1}.
Now $\omega(\cdot,M)$ touches ${\omega}_2$ from above, which is impossible by \autoref{lem:ellipticSMP}. Hence $M\leq 0$. A similar argument interchanging ${\omega}_{1}$ and ${\omega}_2$ implies that ${\omega}_1={\omega}_2$. 
\end{proof}

Next, we demonstrate that small oscillation $\lambda$-STCMCs may always be produced by the flow.  
\begin{prop}\label{prop:SmallOscLTE}
Suppose $\widehat{\Sigma}$ is a $\kappa$-STCMC with graph function $\widehat{\omega}$. Then there exists a $\delta=\delta(\bar{N}, \widehat{\Sigma})>0$ such that the following holds: Let $\widetilde{\Sigma}$ be a hypersurface above $\widehat{\Sigma}$ with 
\eq{\widetilde{\kappa}:=\min_{\widetilde{\Sigma}} |\vec{H}|^2 \in (\kappa, \kappa+1)} and graph function $\widetilde{\omega}$ where $|\widehat{\omega}-\widetilde{\omega}|_{C^0}<\delta$. For $\lambda \in [\kappa,\widetilde{\kappa}]$, let $x^\lambda(t,z)=x(t,z,\lambda)$ be the solutions to \eqref{eq:CMCF} with prescription $\lambda$ starting from $\widetilde{\Sigma}$. Then:
\enum{
\item For every existence time $t$, $z\in \mathcal{S}_0$ and $\lambda\in [\kappa, \widetilde{\kappa}]$, the flow satisfies the a priori estimates \eq{|\omega(t,z,\lambda)-\widehat{\omega}(z)|<\delta\q \text{ and }\q u(t,z,\la)<C(\bar N,\wh{\Si},\widetilde{\Si}).} 
\item The solutions $x^{\la}$ exist for all times $t\in[0,\infty)$ and are uniformly smooth in $t,z$. 
\item The corresponding graph functions $\omega$ are monotonically decreasing in time.
\item The corresponding graph functions $\omega$ are monotonically increasing in $\lambda$.
\item As $t\ra \infty$, each flow $\omega(t, \cdot, \lambda)$ converges to a smooth $\lambda$-STCMC surface given by $\widehat{\omega}_\lambda$. The function $\widehat{\omega}(\cdot, \lambda):=\widehat{\omega}_\lambda$ is monotonically increasing in $\lambda$ and has uniform smooth estimates in $z$ which are independent of $\lambda$.    
}
\end{prop}

\begin{proof}
The key to proving the above is in showing the a priori estimates given in bullet point (i).

Applying the maximum principle to the evolution of $f=\beta-H$ given in \eqref{eq:evolf}, we see that for all times that the flow exists, $f\leq 0$. Hence the flowing graph functions $\omega(t,z,\lambda)$ are non-increasing in time. Furthermore, we observe $0<\omega(t,z,\lambda)-\widehat{\omega}(z)<\delta$: The upper estimate follows by applying \autoref{lem:SMPonomega}, as the flows $x^\lambda$ have $\lambda>\kappa$ 
for all times $t$ and any choice of $\de$. The lower estimate follows as otherwise there would exist a time of first touching $t$ which would contradict \autoref{lem:ellipticSMP}.

Firstly, we pick $\delta_0$ small enough so that 
\eq{\{(s,z)\cn z\in\mathcal{S}_{0}, s=\widehat{\omega}(z)+r, r\in[0,\delta_0]\}\subset\bar{N},}
and consider $\delta\leq \delta_0$. As $\lambda\in[\kappa,\kappa+1)$,  by compactness, the prescription function $\beta=\fr{\lambda}{2} \bar{\theta}^{-1}$ and its derivatives are bounded, specifically, 
\eq{|\beta|+|d\beta|+|\bar{L} \beta|<C_0(\bar{N}, \widehat{\Sigma},\kappa).}
In \autoref{lem:GradEstTestFunctionHelperLemma} we therefore observe that by compactness there is a $C_1=C_1(\bar{N},\widehat{\Sigma})$ (but independent of $\lambda$) so that the functions $\Phi$ and $\Psi$ satisfy
\eq{
|\Phi|\leq c(n)(\bar{\theta}^2+|\ov{\Rc}(\bar{L},\bar{L})|+|\mr{\bar\chi}|^2+\bar{\theta}|\mr{\bar\chi}|)&<C_1\\
|\Psi|\leq c(n)(\bar\theta+|\mr{\bar\chi}|)&<C_1.
}

We consider the test function $\mu(s)=1-a^2s^2$ for $s\in[0,a^{-1})$ where $a>0$ will be determined later. We have
\eq{\mu\leq 1, \qquad 0\geq \mu'\geq -2a, \qquad \mu''=-2a^2,}
so
\eq{\Phi(x)\mu +\Psi(x)\mu'+\mu''\leq C_1+2aC_1-2a^2\leq C_1+C_1^2-a^2.}
 Set $a=\sqrt{C_1+C_1^2+1}$ and choose 
  $\delta=\min\{\frac 1 2 a^{-1}, \delta_0\}$. Due to the bounds on $\omega-\widehat{\omega}$, this means that $\mu$ is smooth and positive on the flowing surface with uniform bounds away from zero and infinity. Then for $\tilde{\phi} = u \mu^{-2}(\omega-\widehat{\omega})$ we have 
\eq{\dot{\tilde{\phi}}-\Delta \tilde{\phi} \leq C(1+\tilde{\phi}^\frac 3 2) - 3\tilde{\phi}^2.}
Hence $\tilde{\phi}$ has no increasing maxima for $\tilde{\phi}$ large enough, and so we have the claimed gradient estimate and we have completed the claim in part (i).

We observe that part (i) implies uniform in $(t,\la)$ $C^{1}$-estimates of $\om(t,\cdot,\la)$. Applying PDE theory, the flow exists for all time, and for all $k$ there is a $C_k$, uniform in $\lambda$ such that 
\eq{|\omega(\cdot,\cdot,\lambda)|_{C^{k+\alpha;\frac{k+\alpha}2}([0,\infty)\times \mathcal{S}_0)}<C_k.} Part (ii) now follows.

As noted earlier, by the maximum principle, $f=\beta-H\leq 0$ for all the time. 
Hence the flows move monotonically as stated in (iii). Furthermore, \autoref{lem:SMPonomega} implies that if $\lambda_1<\lambda_2$ then the flow $x^{\la_{2}}$ is above $x^{\la_{1}}$ at all times, so the flows are monotonically increasing in $\lambda$, from which (iv) follows. 

For each fixed $\lambda$, the flow is monotonically decreasing in time and bounded below, so there must be a limit as $t\ra\infty$. By Arzela--Ascoli this limit must be smooth and the flow must converge uniformly smoothly (by uniform estimates and interpolation). This limit must be a stationary point as otherwise the flow cannot converge. Hence $x(\cdot,t,\lambda)$ converges to a STCMC surface $\widehat{\Sigma}_\lambda$
. Finally, due to the monotonicity of the flow, $\lambda_1<\lambda_2$ implies that $\widehat{\Sigma}_{\lambda_2}$ is above $\widehat{\Sigma}_{\lambda_1}$, completing (v).
\end{proof}

\begin{prop} \label{prop:openness} Suppose that $\widehat{\Sigma}\subset\bar{N}$ is a stable $\kappa$-STCMC hypersurface. Then there exists an $\alpha>0$ such that there exists a continuous $\lambda$-STCMC foliation of a future sided neighbourhood of $\widehat\Si$ for $\la\in [\ka,\ka+\alpha)$.
\end{prop}

\begin{proof}
Let $\wh{\Si} = \mrm{graph}~\wh{\om}$
and suppose that $\gamma=\min\{\epsilon,\delta\}$ where $\epsilon$ is as in \autoref{lem:StabilitytoCMCunique} and $\delta$ is as in \autoref{prop:SmallOscLTE}. 

For $\xi\in(0,\gamma)$ to be determined, define $\tilde{\omega}=\wh{\om} + \xi \frac{f}{|f|_{C^0}}$. By diminishing $\gamma$ further, on $\widetilde{\Sigma}$, there holds
\eq{\ka < \abs{\vec H}^{2}<\ka+1,}
due to the proof of \autoref{lem:StabilitytoCMCunique}. We note, 
\eq{|\tilde{\om}|_{C^3}\leq |\widehat{\om}|_{C^3}+\gamma|f|_{C^3}|f|_{C^0}^{-1}.}
 Furthermore $|\tilde{\om}-\wh{\om}|_{C^0}<\delta$, so, applying \autoref{prop:SmallOscLTE}, we obtain a family of STCMC solutions $\widehat{\omega}(\cdot, \lambda)$ for $\lambda \in [\kappa, \tilde{\kappa}]$ such that $|\widehat{\omega}(\cdot, \lambda)|_{C^{2,\alpha}}<C(\widehat{\omega}, \gamma,f)$. Furthermore, as the $\widehat{\om}(\cdot, \lambda)$ are bounded between $\widehat \om$ and $\tilde{\om}$, $|\widehat{\om}(\cdot, \lambda)-\widehat{\om}|_{C^0}\leq\xi$. Therefore by interpolation, we may choose $\xi$ small enough so that 
\[|\widehat{\om}(\cdot, \lambda)-\widehat{\om}|_{C^2}<\epsilon.\]

We now need to check that there can be no ``gaps'' between the manifolds produced in this way for $\lambda$ small enough.

Fix $\lambda\in (\kappa,\tilde{\kappa})$. Then by monotonicity of $\widehat{\omega}(\cdot, \lambda)$ there are limits 
\eq{\om^{\pm}:=\lim_{i\ra\infty} \omega(\cdot,\lambda\pm i^{-1}).} These limits are smooth by Arzela-Ascoli, and the convergence is smooth by interpolation and uniform estimates. Hence they are both $\lambda$-STCMC surfaces, each of which is $C^2$ $\epsilon$-close to $\widehat{\omega}$, and so by  \autoref{lem:StabilitytoCMCunique}, they are $\widehat{\omega}_\lambda$. 
\end{proof}

\begin{proof}[Proof of \autoref{thm:FoliationExistence}]
 \autoref{prop:openness} implies that given a stable MOTS there is a positive $s>0$ such that the foliation exists with $\lambda$-STCMC leaves $\Sigma_\lambda$ for $\lambda\in [0,s]$.
 We have to show that the foliation is smooth in the sense that $\la\mt \om(\cdot,\la)$ is smooth. We employ the implicit function theorem as in \cite[p.~305]{Gerhardt:/2006d}. There, a smooth operator 
 \eq{G(\la,\wh\om)= \abs{\vec H}^2(\wh\om) - \lambda} 
 is defined, and from the proof of \autoref{lemma:increasing vec H} we observe
 \eq{\del_{\wh\om} G(\wh\om,\tau) = 2\bar\theta \cL_{\Sigma}\wh\om,}
     where $\Sigma = \mrm{graph}(\wh\om)$. We note that the operator $\cL_\Si$ is not self-adjoint. However, as discussed in detail in \cite[Sec.~4, Def.~5.1 and Def.~5.2]{AnderssonMarsSimon:/2008}, under the condition of stability it has a strictly positive smallest real eigenvalue. Hence $\del_{\wh\om} G(\om,\la)$ is invertible and the implicit function theorem shows that the assignment $\la\mt \wh\om(\cdot,\la)$ is smooth and increasing. The foliation is also strictly increasing in the sense of our definition because taking the derivative with respect to $\la$, we obtain
 \eq{0=\del_{\la} G(\la,\wh{\om}_{\la}) = \del_\la \abs{\vec H_{\la}}^2-1 = 2\bar\theta_\la\cL_{\wh\Si_\la}\del_{\la}\wh\om_\la-1<0}
 at a zero minimum of $\del_\la \om_\la$. Hence those zeros can not occur and we conclude $\sigma>0$.

If $|A^{\Sigma_\lambda}|<C$, then by elliptic regularity theory, all higher derivatives are uniformly bounded. Hence we may take a limit to get a $\sigma$-STCMC surface $\Sigma_\sigma$. If $\Sigma_\sigma$ is stable, we may apply \autoref{prop:openness} to see that $\sigma$ was not maximal.

Finally, suppose that $\tilde{\Sigma}_\lambda$ is a smooth $\lambda$-STCMC surface contained in the foliated set for some $\lambda\in [0,\sigma)$. In particular, there is a highest leaf of the foliation with mean curvature $\lambda^\text{high}$ and a lowest leaf of the foliation with mean curvature $\lambda^\text{low}$, which $\Sigma_\lambda$ intersects. \autoref{lem:ellipticSMP} implies that if $\ti\Sigma_\lambda$ is not a leaf of the foliation, then $\lambda^\text{high}<\lambda$ and $\lambda<\lambda^\text{low}$, which is impossible as $\lambda^\text{low}\leq\lambda^\text{high}$. Hence $\ti\Sigma_\lambda$ is a leaf of the foliation. 
\end{proof}
\section{The prescribed mean curvature problem and proof of \autoref{thm:PrescribedMeanCurvature}} 

The proof of \autoref{thm:PrescribedMeanCurvature} proceeds by providing $C^{1}$-estimates, from which everything else follows from parabolic regularity as in \autoref{sec:foliation}. First of all we note that from the maximum principle, the hypersurfaces $\cS_{0}$ and $\Si^{+}$ are barriers for the flow.

The key to the $C^{1}$-estimates is evident from the evolution equation of $\phi$ in \autoref{lem:GradEstTestFunctionHelperLemma}. A sufficient ingredient for obtaining a bound on $u$ from this lemma is to find a positive test function $\mu = \mu(\om)$ with the property
\eq{\Phi \mu(\omega) +\Psi \mu'(\omega)+\mu''(\omega)<0.\label{eq:KeyInequality}}
From the structure of this ODE it is evident that finding a {\it positive} solution can potentially be hampered by $\mu''$ having to be very negative. Hence, in some cases, the allowed range for $\om$ has to be restricted. We give the details in the following and prove that under the current conditions, the required test functions can be found. Before we can do so, we have to control $\bar\theta$ on $\bar N$ given the validity of \eqref{eq:Nestimatedbytheta}.

\begin{lemma}\label{lem:thetabarest}
Suppose that equation \eqref{eq:Nestimatedbytheta} holds on $\bar{N}$.
Then for all $s\in [0,\La)$,
\eq{\frac{(n^{-1}+C_R+c_{\mr{\bar{\chi}}}^2)^{-1}}{(n^{-1}+C_R+c_{\mr{\bar{\chi}}}^2)^{-1}+s} \leq\bar{\theta}(s,\cdot)\leq \frac{(n^{-1}+c_R)^{-1}}{(n^{-1}+c_R)^{-1}+s}.}
\end{lemma}
\begin{proof}
By the Raychaudhuri equations in \autoref{prop:optical} we know that
\eq{-(n^{-1}+C_R+c_{\mr{\bar{\chi}}}^2) \bar{\theta}^2\leq\partial_s (\bar\theta)=\bar{L}(\bar{\theta})= -\tfrac 1 n \bar{\theta}^2-|\mr{\bar\chi}|^2-\ov{\Rc}(\bar{L},\bar{L})\leq -(n^{-1}+c_R) \bar{\theta}^2.}
Hence,
\eq{(n^{-1}+c_R) \leq\partial_s (\bar{\theta}^{-1})\leq (n^{-1}+C_R+c_{\mr{\bar{\chi}}}^2).}
As $\bar{\theta}(0,\cdot)=1$, by integrating we obtain
\eq{\frac{(n^{-1}+C_R+c_{\mr{\bar{\chi}}}^2)^{-1}}{(n^{-1}+C_R+c_{\mr{\bar{\chi}}}^2)^{-1}+s} \leq\bar{\theta}(s,\cdot)\leq \frac{(n^{-1}+c_R)^{-1}}{(n^{-1}+c_R)^{-1}+s}.}
\end{proof}

Given this control on $\bar\theta$, we can solve \eqref{eq:KeyInequality}.

\begin{lemma}\label{lem:ODIsolution}
Suppose the validity of \eqref{eq:Nestimatedbytheta} on $\bar{N} = [0,\La)\x \cS_{0}$
and suppose that we are given an interval $[a,b]\subset [0,\Lambda)$. Additionally, suppose that one of the following two cases hold: 
\begin{enumerate}
\item Either $n\leq 6$ and
\eq{c_{\mr{\bar{\chi}}}\leq \frac{c_R}{2}+\frac{6-n}{4n},\label{eq:smallnassumption}}
\item or $n\geq 7$ and 
\eq{2c_{\mr{\bar{\chi}}}+c_{\mr{\bar{\chi}}}^2+C_R\leq \frac{n-6}{2n}.\label{eq:largenassumption}}
\end{enumerate}
Then there is an explicit constant $\mathcal{D}(n,c_R,c_{\mr{\bar{\chi}}}, C_R-c_R)$ which is smooth in its last three entries where \[\mathcal{D}(n,c_R,0,0)= \tfrac{n^2}{(1+nc_R)^2}\left(\tfrac{(n-2)(n-10)}{4n^2} + \tfrac{n+6}{n}c_R+c_R^2\right)\ ,\] 
such that the following holds:
\begin{itemize}
\item If $\mathcal{D}\geq 0$ then there is a positive solution $\mu\cn [a,b]\ra \mathbb{R}$ to $\eqref{eq:KeyInequality}$.
\item If $\mathcal{D}< 0$ and $b<\frac{n}{(1+nc_R)}\left(e^{\frac{\pi}{\sqrt{-\mathcal{D}}}}-1\right)$, then there exists a positive solution $\mu\cn [a,b]\ra \mathbb{R}$ to \eqref{eq:KeyInequality}.
\end{itemize}

\end{lemma}

\begin{proof}
Our aim will be to solve \eqref{eq:KeyInequality} by comparing this with solutions of the Euler-Cauchy equations, using our estimates on $\bar{\theta}$ from \autoref{lem:thetabarest} to estimate the coefficients in \eqref{eq:KeyInequality}. 

If $\Phi<0$ then it suffices to pick $\mu=1$. Hence, from now on we assume that there are points with $\Phi\geq 0$.
Set
\eq{W:=(n^{-1}+c_{R})^{-1}=\frac{n}{1+nc_R},\qquad Z:=(n^{-1}+C_R+c_{\mr{\bar{\chi}}}^2)^{-1}=\frac{n}{1+n(C_R+c_{\mr{\bar{\chi}}}^2)}\leq W}
so that from \autoref{lem:thetabarest} we see that
\eq{Z\leq (s+Z)\bar\theta\leq (s+W) \bar{\theta}\leq W.\label{eq:GetRidofthetabar}}

Suppose first that $n\leq 6$. We define the constants
\eq{B_1:=\begin{cases} W\left(\frac{n-4}{2n}+2c_{\mr{\bar{\chi}}}\right)& \text{if }\frac{n-4}{2n}+2c_{\mr{\bar{\chi}}}>0\\ Z\left(\frac{n-4}{2n}+2c_{\mr{\bar{\chi}}}\right)&\text{otherwise}, \end{cases}}
\eq{A_\delta:=W^2\left(n^{-2}-\frac{c_R}{2}+\frac{3}{2}c_{\mr{\bar{\chi}}}^2+\frac{|4-n|}{2n}c_{\mr{\bar{\chi}}}\right)+\delta,}
where $\delta\geq 0$ will be chosen later. Note that $A_{\de}$ has been chosen so that
\eq{\Phi\leq \frac{A_{\de}-\delta}{(s+W)^2},}
so as $\Phi\geq 0$, we see that $A_\delta\geq \delta$.
We will shortly find $\mu = \mu(s)>0$ which solves
\eq{0=\frac{A_\delta}{(s+W)^2}\mu + \frac{B_1}{s+W} \mu' +\mu''.\label{eq:EulerCauchyEq}}
Given such $\mu$, by our choice of $A_{\de}$ and the positivity of $\mu$, we then estimate
\eq{\Phi \mu+\Psi \mu'+\mu'' \leq \frac{A_\delta-\delta}{(s+W)^2}\mu+\Psi \mu'+\mu''= ((s + W)\Psi-B_1)\frac{\mu'}{s+W}-\delta \frac{\mu}{(s+W)^2} .\label{eq:stepinproof1}}
Note that by our choice of $B_1$ and \eqref{eq:GetRidofthetabar}, 
\eq{(s + W)\Psi-B_1\leq (\tfrac{n-4}{2n}+2c_{\mr{\bar{\chi}}})(s+W)\bar{\theta}-B_1\leq 0.}
Hence, our aim is to show that we may find a solution $\mu$ to \eqref{eq:EulerCauchyEq} with $\delta>0$,  $\mu>0$ and $\mu'>0$, so that the right hand side of \eqref{eq:stepinproof1} is negative. 

We note that by our assumption \eqref{eq:smallnassumption}
\eq{B_1\leq (n^{-1}+c_R)W=1.}
We set $D_{1,\delta}=(B_1-1)^2-4A_\delta$ then:
\begin{itemize}
\item If $D_{1,\delta} \geq 0$, let $\mu\cn[0,\infty) \ra \mathbb{R}$ be given by
\eq{\mu(s)=(s+W)^{\frac 1 2(1-B_1+\sqrt{D_{1,\delta}})}}
which solves \eqref{eq:EulerCauchyEq} and satisfies the required properties $\mu>0$ and $\mu'>0$, because $B_1\leq 1$.
\item If $D_{1,\delta} <0$, then for some $\eta\in(0, \tfrac \pi 2)$ define 
\[I_{1,\delta,\eta}=[0,W(e^{(\pi-2\eta)/\sqrt{-D_{1,\delta}}}-1))\] 
and let $\mu\cn I_{1,\delta,\eta}\ra \mathbb{R}$ be given by
\eq{\mu(s)=(s+W)^{\frac 12(1-B_1)}\sin\left(\tfrac{\sqrt{-D_{1,\delta}}}{2}\log(W^{-1}s+1)+\eta\right),}
which solves \eqref{eq:EulerCauchyEq} and, as $B_1 \leq 1$, has both $\mu>0$ and $\mu'>0$ on $I_{1,\delta,\eta}$.
\end{itemize}
We now choose $\delta>0$ to ensure the strict inequality on the claimed interval. We define 
\[\mathcal{D}=\lim_{\delta\ra0} D_{1,\delta}=(B_1-1)^2-4A_0\] and note that $D_{1,\delta}$ is monotonically decreasing in $\delta$. 
\begin{itemize}
\item If $\mathcal{D}>0$, pick $\delta>0$ sufficiently small that $D_{1,\delta}>0$, then $\mu$ as described above will solve \eqref{eq:KeyInequality}. 
\item If $\cD\leq 0$, then by a suitable choice of $\de$ we can achieve $W(e^{\pi/\sqrt{-D_{1,\delta}}}-1)>b$. As the ends of the interval vary continuously with $\eta$, we may pick $\eta>0$ small enough so that $[0,b]\subset I_{1,\delta,\eta}$.  Given these choices, $\mu$ restricted to $[a,b]$ with $D_{1,\delta}<0$ as above satisfies \eqref{eq:KeyInequality} as required.
\end{itemize}

Now suppose that $n\geq 7$. Our aim is to follow an identical process with $A_\delta$, $W$ and $Z$ defined as above, 
but we will replace $B_1$ with
\eq{B_2:=Z\left(\frac{n-4}{2n}-2c_{\mr{\bar{\chi}}}\right).}
Similarly to the previous cases we will shortly choose $\mu>0$ to be a solution of 
\eq{0=\frac{A_\delta}{(s+W)^2}\mu + \frac{B_2}{s+W} \mu' +\mu'',\label{eq:EulerCauchyEq2}}
so that, as in \eqref{eq:stepinproof1},
\eq{\Phi \mu+\Psi \mu'+\mu'' \leq  ((s+W)\Psi-B_2)\frac{\mu'}{s+W}-\delta \frac{\mu}{(s+W)^2} .\label{eq:stepinproof2}}
Note that \eqref{eq:largenassumption} implies that $2c_{\mr{\bar{\chi}}}\leq \tfrac{n-4}{2n}$. Hence using equation \eqref{eq:GetRidofthetabar},
\eq{(s+W)\Psi-B_2\geq (\tfrac{n-4}{2n}-2c_{\mr{\bar{\chi}}})(s+W)\bar{\theta}-B_2\geq \left(\tfrac{n-4}{2n}-2c_{\mr{\bar{\chi}}}\right)Z-B_2=0.}
Therefore this time, we search for solutions of \eqref{eq:EulerCauchyEq2} with $\delta>0$, $\mu>0$ and $\mu'<0$ to render the right hand side of \eqref{eq:stepinproof1} negative. 

We note that \eqref{eq:largenassumption} implies
\eq{B_2\geq Z\left(\frac 1 n +c_{\mr{\bar{\chi}}}^2+C_R\right)=1.}
We set $D_{2,\delta}=(B_2-1)^2-4A_\delta$ then:
\begin{itemize}
\item If $D_{2,\delta} \geq 0$, let $\mu\cn[0,\infty) \ra \mathbb{R}$ be given by
\eq{\mu(s)=(s+W)^{\frac 1 2(1-B_2+\sqrt{D_{2,\delta}})},}
which solves \eqref{eq:EulerCauchyEq} and satisfies the required properties $\mu>0$ and $\mu'<0$.
\item If $D_{2,\delta} <0$, then for some $\eta\in (0, \tfrac\pi 2)$ define 
\[I_{2,\delta,\eta}=[0,W(e^{(\pi-2\eta)/\sqrt{-D_{1,\delta}}}-1))\] 
and let $\mu\cn I_{2,\delta,\eta}\ra \mathbb{R}$ be given by
\eq{\mu(s)=(s+W)^{\frac 12(1-B_2)}\cos\left(\tfrac{\sqrt{-D_{1,\delta}}}{2}\log(W^{-1}s+1)+\eta\right),}
which solves \eqref{eq:EulerCauchyEq} and, as $B_2\geq 1$, has both $\mu>0$ and $\mu'<0$ on the stated interval.
\end{itemize}
As in the previous case, set $\mathcal{D}:=\lim_{\delta\ra 0} D_{2,\delta}=(B_2-1)^2-4A_0$ and note that $D_{2,\delta}$ is monotonically decreasing in $\delta$. Picking $\delta,\eta>0$, a continuity argument as in the case $n\leq 6$ completes the proof.

Finally, we check that in either of the cases, $\mathcal{D}$ has the claimed form. Note that $W$, $Z$, $B_1$, $B_2$, $A_0$ are all smooth functions of $c_R$, $C_R-c_R$ and $c_{\mr{\bar{\chi}}}$ (for $0\leq c_R, C_R-c_R, c_{\mr{\bar{\chi}}}$), and hence, $\mathcal{D}$ is smooth in each of its last three entries. If $C_R-c_R=c_{\mr{\bar{\chi}}}=0$ then $W=Z=\frac{n}{1+nc_R}$, $B_1=B_2=\frac{n-4}{2n}W$ and $A_0=W^2(n^{-2}-\frac{c_R}{2})$. Hence
\begin{flalign*}
W^{-2}\mathcal{D} &= (W^{-1}B_1-W^{-1})^2-4W^{-2}A_0= (\tfrac{n-6}{2n} -c_R)^2-4n^{-2}+2c_R\\
&=\frac{(n-2)(n-10)}{4n^2}+\frac{6+n}{n}c_R+c_R^2,
\end{flalign*}
as claimed.
\end{proof}

We proceed with the $C^{1}$-estimates.

\begin{lemma}
Under the assumptions of \autoref{thm:PrescribedMeanCurvature}, the flow \eqref{eq:PMCF} with $\be = \rho/(2\bar\theta)$ satisfies uniform $C^{1}$-estimates during the evolution and, in turn, satisfies smooth estimates and converges to a solution to
\eq{\abs{\vec H}^{2} = \rho.}
\end{lemma}

\pf{We pick $C_0(n)$ and $C_1(n)$ so that conditions \eqref{eq:smallnassumption} and \eqref{eq:largenassumption} in \autoref{lem:ODIsolution} hold. 
By the maximum principle, the flow remains between $\mathcal{S}_0$ and $\Sigma^+$. 
The assumption on $\Si^{+}$ is precisely there to ensure that, using \autoref{lem:ODIsolution}, we can build a test function $\phi = u\mu^{-2}(\om)$ as in \autoref{lem:GradEstTestFunctionHelperLemma}, which ensures a strictly negative sign on the highest order term in the evolution equation of $\phi$, which in turn yields
a gradient bound on the flowing manifolds. Here we also crucially use that the flow remains in the region between $\cS_{0}$ and $\Si^{+}$. Using the $C^{1}$-bounds, standard bootstrapping gives smooth estimates and, from the monotonicity of the flow, converges to a stationary limit of the flow. 
}

\begin{rem}\label{rmk:foliationsize}
Suppose that $\mathcal{S}_0$ is a MOTS. Then for $c_{\mr{\bar{\chi}}}, c_R, C_R, \mathcal{D}$ as in \autoref{thm:PrescribedMeanCurvature}. Define 
\eq{M=\begin{cases}
\Lambda &\text{ if }\mathcal{D}\geq 0\\
\min\left\{\frac{n}{1+nc_R}(e^\frac{\pi}{\sqrt{-\mathcal{D}}}-1), \Lambda\right\} & \text{ otherwise.}
\end{cases}}
Then, either
\begin{itemize}
\item the foliation of \autoref{thm:FoliationExistence} leaves every compact subset of $\mathcal{S}_0\times[0,M)$, or
\item there is a smooth unstable $\sigma$-STCMC $\Sigma_\sigma\subset \mathcal{S}_0\times[0,M)$ and the region between $\mathcal{S}_0$ and $\Sigma_\sigma$ is smoothly foliated by STCMC hypersurfaces.
\end{itemize}
\end{rem}
\begin{proof}
In the construction, replace \autoref{prop:SmallOscLTE}(i) with the following gradient estimate, which depends only on $c_{\mr{\bar{\chi}}}, C_R, C_R-c_R, \mathcal{D}$, but is independent of initial data.

As in the previous Lemma, we can build a test function $\phi = u\mu^{-2}(\om)$ as in \autoref{lem:GradEstTestFunctionHelperLemma}, with evolution given by
\[\dot{\phi} - \Delta \phi = -2c\phi^2+C(\phi^\frac 32+1)\ . \]
for some constants $c,C>0$ depending only on $c_{\mr{\bar{\chi}}}, C_R, C_R-c_R, \mathcal{D}$. We start by estimating $\phi$ for small times. By ODE comparison, we have $\phi\leq \Theta(t)$ where 
\[\dot\Theta\leq -2c\Theta^2+C(\Theta^\frac 32+1)\leq -c\Theta^2+C\]
 and $\Theta(0)=\max_{M_0} \phi$. Define $\tilde{\Theta}:=\max\left\{0, \Theta - 2\sqrt{\frac{C}{2c}}\right\}$ so that when the derivative exists,
 \eq{\partial_t\tilde{\Theta}\leq \begin{cases}
 -2c\Theta^2 +C, & \text{if}~\Theta>2\sqrt{\frac{C}{2c}}\\
 0&\text{otherwise}
 \end{cases} 
 &\leq\begin{cases}
 -c\Theta^2, & \text{if}~\Theta>2\sqrt{\frac{C}{2c}}\\
 0&\text{otherwise}
 \end{cases}\\
 &\leq \begin{cases}
 -c\tilde{\Theta}^2, &\text{if}~ \Theta>2\sqrt{\frac{C}{2c}}\\
 0&\text{otherwise}
 \end{cases}\\
 &=-c\tilde{\Theta}^2,}
 so solving this we may estimate
\[\phi(\cdot,t)\leq 2\sqrt{\frac{C}{2c}} + \frac{1}{ct},\]
independently of initial data. Therefore, we have an estimate on $\phi$ at time $t=1$ which is independent of initial data. Applying the maximum principle beyond this point implies a bound on $\phi$ depending only on $c_{\mr{\bar{\chi}}}, C_R, C_R-c_R, \mathcal{D}$ for $t\geq 1$. Standard PDE estimates imply $C^k$ estimates on the flow. 

Hence the limit surfaces $\Sigma_\lambda$ satisfy uniform smooth estimates independently of initial data. 
Hence if the foliation doesn't leave every compact set of $\mathcal{S}_0\times [0,M)$, then all leaves satisfy uniform smooth bounds. Hence the only possibility is that the foliation terminates in a smooth unstable $\sigma$-STCMC hypersurface. 
\end{proof}

\providecommand{\bysame}{\leavevmode\hbox to3em{\hrulefill}\thinspace}
\providecommand{\MR}{\relax\ifhmode\unskip\space\fi MR }
\providecommand{\MRhref}[2]{%
  \href{http://www.ams.org/mathscinet-getitem?mr=#1}{#2}
}
\providecommand{\href}[2]{#2}


\begin{thebibliography}{10}

\bibitem{AnderssonMarsSimon:/2008}
Lars Andersson, Marc Mars, and Walter Simon, \emph{Stability of marginally
  outer trapped surfaces and existence of marginally outer trapped tubes}, Adv.
  Theo. Math. Phys. \textbf{12} (2008), no.~4, 853--888.

\bibitem{BourniMoore:/2019}
Theodora Bourni and Kristen Moore, \emph{Null mean curvature flow and outermost
  {MOTS}}, J. Differ. Geom. \textbf{111} (2019), no.~2, 191--239.

\bibitem{BrendleEichmair:09/2014}
Simon Brendle and Michael Eichmair, \emph{Large outlying stable constant mean
  curvature spheres in initial data sets}, Invent. Math. \textbf{197} (2014),
  no.~3, 663--682.

\bibitem{CederbaumSakovich:/2021}
Carla Cederbaum and Anna Sakovich, \emph{On center of mass and foliations by
  constant spacetime mean curvature surfaces for isolated systems in general
  relativity}, Calc. Var. Partial Differ. Equ. \textbf{60} (2021), no.~6, art.
  214.

\bibitem{EichmairKoerber:11/2024}
Michael Eichmair and Thomas Koerber, \emph{Foliations of asymptotically flat
  manifolds by stable constant mean curvature spheres}, J. Differ. Geom.
  \textbf{128} (2024), no.~3, 1037--1083.

\bibitem{Gerhardt:12/1983}
Claus Gerhardt, \emph{H-surfaces in {L}orentzian manifolds}, Commun. Math.
  Phys. \textbf{89} (1983), no.~4, 523--553.

\bibitem{Gerhardt:/2000}
\bysame, \emph{Hypersurfaces of prescribed curvature in {L}orentzian
  manifolds}, Indiana Univ. Math. J. \textbf{49} (2000), no.~3, 1125--1153.

\bibitem{Gerhardt:09/2000}
\bysame, \emph{Hypersurfaces of prescribed mean curvature in {L}orentzian
  manifolds}, Math. Z. \textbf{235} (2000), no.~1, 83--97.

\bibitem{Gerhardt:/2006d}
\bysame, \emph{On the {CMC} foliation of future ends of a spacetime}, Pac. J.
  Math. \textbf{226} (2006), no.~2, 297--308.

\bibitem{Huang:12/2010}
Lan-Hsuan Huang, \emph{Foliations by stable spheres with constant mean
  curvature for isolated systems with general asymptotics}, Commun. Math. Phys.
  \textbf{300} (2010), no.~2, 331--373.

\bibitem{HuiskenWolff:07/2025}
Gerhard Huisken and Markus Wolff, \emph{On the evolution of hypersurfaces along
  their inverse space-time mean curvature}, J. Differ. Geom. \textbf{130}
  (2025), no.~3, 571--633.

\bibitem{HuiskenYau:/1996}
Gerhard Huisken and Shing-Tung Yau, \emph{Definition of center of mass for
  isolated physical systems and unique foliations by stable spheres with
  constant mean curvatre}, Invent. Math. \textbf{124} (1996), 281--311.

\bibitem{KronckeWolff:12/2024}
Klaus Kr\"oncke and Markus Wolff, \emph{Foliations of asymptotically
  {S}chwarzschildean lightcones by surfaces of constant spacetime mean
  curvature}, {\href{https://arxiv.org/abs/2412.17563}{arxiv:2412.17563}}, 12
  2024.

\bibitem{Lieberman:/1998}
Gary Lieberman, \emph{Second order parabolic differential equations}, World
  Scientific, Singapore, 1998.

\bibitem{Ma:10/2011}
Shiguang Ma, \emph{Uniqueness of the foliations of constant mean curvature
  spheres in asymptotically flat 3-manifolds}, Pac. J. Math. \textbf{252}
  (2011), no.~1, 145--179.

\bibitem{ONeill:/1983}
Barrett O'Neill, \emph{Semi-{R}iemannian geometry with applications to
  relativity}, Pure and applied mathematics, vol. 103, Academic Press, San
  Diego, 1983.

\bibitem{RoeschScheuer:02/2022}
Henri Roesch and Julian Scheuer, \emph{Mean curvature flow in null
  hypersurfaces and the detection of {MOTS}}, Commun. Math. Phys. \textbf{390}
  (2022), no.~3, 1149--1173.

\bibitem{Tod:/1991}
Paul Tod, \emph{Looking for marginally trapped surfaces}, Class. Quantum Grav.
  \textbf{8} (1991), no.~5, 115--118.

\bibitem{Wolff:04/2023}
Markus Wolff, \emph{Ricci flow on surfaces along the standard lightcone in the
  $3+1$-{M}inkowski spacetime}, Calc. Var. Partial Differ. Equ. \textbf{62}
  (2023), no.~3, art. 90.

\end{thebibliography}

\end{document}